\documentclass{article}

\usepackage{amsfonts,amsthm,amsmath,mathrsfs,graphicx,enumitem,amssymb}
\usepackage[utf8]{inputenc} 

\newtheorem{thm}{Theorem}
\newtheorem{lem}{Lemma}
\newtheorem{cor}{Corollary}
\newtheorem{prop}{Proposition}

\begin{document}



\title{Kolmogorov type inequalities for the Marchaud fractional derivatives on the real line and the half-line}
\author{V.\,F.~Babenko, M.\,S.~Churilova, \\N.\,V.~Parfinovych, D.\,S.~Skorokhodov}
\date{Department of Mechanics and Mathematics, \\Dnepropetrovsk National University}
\maketitle



\begin{abstract}
In this paper we establish some new Kolmogorov type inequalities for the Marchaud and Hadamard fractional derivatives of functions defined on a real axis or semi-axis. Simultaneously we solve two related problems: the Stechkin problem on the best approximation of unbounded operators by bounded ones on a given class of elements and the problem of optimal recovery of operator on elements from some class given with prescribed error.


\end{abstract}





\section{Introduction}
Inequalities estimating the norm of intermediate derivative of a function in terms of the norm of the function itself and the norm of its higher order derivative (inequalities of Kolmogorov type) are important in many areas of Mathematics and its applications. Due to efforts of many mathematicians, nowadays, great number of sharp Kolmogorov type inequalities is known (see, for instance, surveys~\cite{ArGab,Arestov_96,Kolm_85} and monographs~\cite{BKKP-kn,Kwong,Mitr}). In many questions of Analysis and its Applications the study of fractional order derivatives is also important (see, for instance,~\cite{Sam}). For some known results on the Kolmogorov type inequalities for derivatives of fractional order we refer the reader to the papers~\cite{geisb,Ge68,arestov,Mag_Tih_81,bch,Bab_Chur_07,Bab_Pich_10,Bab_Parf_Pich_14}, book~\cite[Ch.~2]{MBDK} and references therein.

In this paper we shall obtain some new Kolmogorov type inequalities for fractional derivatives. Simultaneously, we consider two closely related problems: the Stechkin problem on approximation of unbounded operators by bounded ones on a given class of elements $Q$, and the problem of optimal recovery of unbounded operator on the class $Q$ under assumption that elements in $Q$ are given with known error (for more information see~\cite{ArGab,Arestov_96} and~\cite[\S 7.1]{BKKP-kn}).

\subsection{The Kolmogorov type inequalities}
Let $G$ be the real line $\mathbb{R} = (-\infty, +\infty)$ or half-line $\mathbb{R}_+ = [0,+\infty)$. By $L_{p}\left( G \right)$, $1 \leqslant p \leqslant \infty$, we denote the space of measurable functions $f:G\to\mathbb{R}$ whose modulus to the $p$-th power is integrable on $G$ (essentially bounded on $G$ if $p=\infty$), endowed with the standard norm
\[
	\|f\|_{L_{p}\left( G \right)}:= \left\{ \begin{array}{ll}
		\displaystyle \left(\int_G\left|f(t)\right|^{p}\,dt \right)^{1/p}, & \textrm{if} \;1\leqslant p < \infty,\\[10pt]
		\textrm{ess\,sup}\left\{|f(t)|\;:\;t\in G\right\},&\textrm{if}\; p=\infty.
	\end{array}\right.
\]
For $r\in\mathbb{N}$ and $1\leqslant s\leqslant \infty$, by $L_{p,s}^r(G)$ we denote the space of functions $f\in L_{p}(G)$ having locally absolutely continuous on $G$ derivative $f^{(r-1)}$ and such that $f^{(r)}\in L_s(G)$.

Let $1\leqslant q\leqslant \infty$, $k\in\mathbb{N}\cup\{0\}$, $0\leqslant k\leqslant r-1$, and $\lambda,\mu\in\mathbb{R}$. Inequalities of the form
\begin{equation}
\label{multiplicative_inequality}
	\left\|f^{(k)}\right\|_{L_{q}(G)} \leqslant K \left\|f\right\|_{L_{p}(G)}^{\mu} \left\|f^{(r)}\right\|_{L_{s}(G)}^{\lambda}
\end{equation}
holding true for every function $f\in L_{p,s}^{r}(G)$ with some constant $K$ independent of $f$ are called {\it the Kolmogorov type inequalities} ({\it the Kolmogorov-Nagy type inequalities} when $k=0$). It is known (see~\cite{Gab_67}) that the constant $K$ in inequality~(\ref{multiplicative_inequality}) is finite if and only if
\begin{equation}
\label{Gabushin_conditions1}
	\lambda = \frac{k-1/q + 1/p}{r- 1/s + 1/p},\qquad\mu = 1-\lambda
\end{equation}
and
\[
	\frac{r}{q} \leqslant \frac{r-k}{p} + \frac{k}{s}.
\]
Naturally, inequalities with the lowest possible (sharp) constant $K$ are of the most interest. We refer the reader to the papers~\cite{ArGab,Arestov_96} and books~\cite{BKKP-kn,Mitr} for the detailed survey on the Kolmogorov type inequalities and discussion of related questions.

Together with inequalities~(\ref{multiplicative_inequality}) the study of inequalities between the norms of intermediate function derivative, the function itself and its higher order derivative in spaces more general than $L_p$ are also important. In Sections~\ref{Sec:aux}--\ref{Sec:Low_smoothness} we shall obtain several inequalities between the norms of derivatives in ideal lattices (see~\cite[Ch.~2, \S2]{Krein}).

In this paper we focus on the study of the Kolmogorov type inequalities for non-integer (fractional) values of $k$. There are many ways to give a sense to the fractional derivative of a function defined on $\mathbb{R}$ or $\mathbb{R}_+$. Among the first ones was the fractional derivative in {\it the Riemann-Liouville} sense (see~\cite[\S5.1]{Sam}) that is defined for a function $f:\mathbb{R}\to\mathbb{R}$ and $x\in\mathbb{R}$, as follows
\begin{equation}
\label{Riemann-Liouville}
	\mathcal{D}^{k}_{\pm}f(x) := \frac{(\pm 1)^{n}}{\Gamma(n-k)}\cdot\frac{d^{n}}{dx^{n}}\int_0^{+\infty} t^{n-k-1} f(x\mp t)\,dt,\qquad n=[k]+1,
\end{equation}
where $\Gamma(z)$ is the Euler gamma function and $[z]$ stands for the integer part of real number $z$.
We shall mostly consider fractional derivatives in {\it the Marchaud} sense (see~\cite{march} or~\cite[\S5.6]{Sam}) that are defined for a function $f:\mathbb{R}\to\mathbb{R}$ and $x\in\mathbb{R}$, as follows
\begin{equation}
\label{Marchaud_derivative}
	D_{\pm}^{k}f(x) = \frac{1}{\varkappa(k,n)} \int_{0}^{+\infty} \frac{\left(\Delta^n_{\pm t}f\right)(x)}{t^{1+k}}\,dt
\end{equation}
where $n\in\mathbb{N}$, $n>k$, (the definition itself is independent of $n$) and
\begin{equation}
\label{difference}	
	\begin{array}{rcl}
		\displaystyle\left(\Delta^n_{\pm t}f\right)(x) & := & \displaystyle\sum\limits_{m=0}^{n} (-1)^m\left(n\atop m\right) f\left(x \mp mt\right), \\
		\displaystyle\varkappa(k,n) & := & \displaystyle\Gamma(-k) \sum\limits_{m=0}^{n}(-1)^{m}\left(n\atop m \right) m^{k}.
	\end{array}
\end{equation}
For a function $f:\mathbb{R}_+\to\mathbb{R}$, the right hand sided derivatives $\mathcal{D}^{k}_{-}f$ and $D^{k}_{-}f$ are defined by formulas~(\ref{Riemann-Liouville}) and~(\ref{Marchaud_derivative}) respectively. The left hand sided derivatives $\mathcal{D}^{k}_{+}f$ and $D^{k}_{+}f$ are defined with the help of slightly different constructions (see~\cite[\S\S5.1, 5.5]{Sam}), and we shall not study these derivatives here.

It is known (see~\cite{Sam}) that $\mathcal{D}^k_{\pm}f = D^k_{\pm}f$ for ``good'' functions $f:G\to\mathbb{R}$. However, construction~(\ref{Marchaud_derivative}) is also suitable for a wider class of functions, e.g. constant functions or functions whose power growth at infinity has order lower than $k$.

Let us consider Kolmogorov type inequalities of the form~(\ref{multiplicative_inequality}) with the term $\left\|f^{(k)}\right\|_{L_q(G)}$ being replaced by $\left\|D^{k}_-f\right\|_{L_q(G)}$ in the left hand side:
\begin{equation}
\label{multiplicative_inequality_fractional}
	\left\|D^k_-f\right\|_{L_{q}(G)} \leqslant K \left\|f\right\|_{L_{p}(G)}^{\mu} \left\|f^{(r)}\right\|_{L_{s}(G)}^{\lambda}, \qquad f\in L^r_{p,s}(G).
\end{equation}
Similarly to inequalities~(\ref{multiplicative_inequality}) for derivatives of integer order, it is easy to see that the constant $K$ in~(\ref{multiplicative_inequality_fractional}) is finite only if parameters $\lambda$ and $\mu$ satisfy equalities~(\ref{Gabushin_conditions1}). 

Together with the Riemann-Liouville and the Marchaud fractional derivatives, Kolmogorov type inequalities were also studied (\cite{Bab_Parf_12,Bab_Parf_Pich_14,bustih}) for other fractional derivatives, e.g. the Riesz fractional derivative, the Weyl fractional derivative, etc. To the best of our knowledge, sharp constant in inequality~(\ref{multiplicative_inequality_fractional}) was found in the following situations:
\begin{enumerate}
\item $G=\mathbb{R}$, $p=q=s=\infty$, $r=2$ and $k\in(0,1)$, -- S.\,P.~Geisberg~\cite{geisb};
\item $G = \mathbb{R}_+$, $p=q=s=\infty$, $r=2$ and $k\in(0,2)\setminus\{1\}$, -- V.\,V.~Arestov~\cite{arestov} (for $\mathcal{D}^{k}_{\pm}$);
\item $G = \mathbb{R}_+$, $p=q=s=\infty$, $0<k\leqslant 1$ and $k<r\leqslant 2$, -- V.\,V.~Arestov~\cite{arestov} (for $\mathcal{D}^{k}_{\pm}$);
\item $G = \mathbb{R}$ or $G = \mathbb{R}_+$, $p=s=2$, $q=\infty$, $k<r$, -- A.\,P.~Buslaev and V.\,M.~Tihomirov~\cite{bustih} (for the Weyl derivative);
\item $G = \mathbb{R}$ or $G = \mathbb{R}_+$, $p=q=\infty$, $1\leqslant s\leqslant \infty$, $r=1$, $k\in\left(0,1-1/s\right)$, -- V.\,F.~Babenko and M.\,S.~Churilova~\cite{bch_3};
\item $G = \mathbb{R}$, $p=q=\infty$, $1\leqslant s\leqslant \infty$, $r=1,2$, $k\in\left(0,r-1/s\right)\setminus\{1\}$, -- V.\,F.~Babenko, N.\,V.~Parfinovych~\cite{Bab_Parf_12} and V.\,F.~Babenko, N.\,V.~Parfinovych, S.\,A.~Pichugov~\cite{Bab_Parf_Pich_14} (for the Riesz derivative).
\end{enumerate}

Here we establish sufficient conditions which allow writing sharp Kolmogorov type inequalities. Specifically, we focus on inequalities between the uniform norms of the function and its derivatives; the uniform norms of the function and its intermediate derivative and the norm of its higher order derivative in the ideal lattice; the norms of the function and its derivatives in the ideal lattice. As a consequence in Section~\ref{Sec:Low_smoothness} we obtain several new sharp Kolmogorov type inequalities in the following cases:
\begin{enumerate}
\item $G = \mathbb{R}$ or $G = \mathbb{R}_+$, $p=q=\infty$, $r=1$ and $k\in(0,1)$, the norm of $f'$ is considered in an ideal lattice;
\item $G = \mathbb{R}$, $p=q=s=1$, $r=1$ and $k\in(0,1)$;
\item $G = \mathbb{R}_+$, $p=q=\infty$, $r=2$ and $k\in(0,1)$, the norm of $f''$ is considered in an ideal lattice;
\item $G = \mathbb{R}$, $p=q=\infty$, $1\leqslant s\leqslant \infty$, $r=2$ and $k\in(0,1)$;
\item $G = \mathbb{R}$ or $G = \mathbb{R}_+$, $p=q=\infty$, $1\leqslant s\leqslant \infty$, $r=2$ and $k\in\left(1,2-1/s\right)$.
\end{enumerate}

\subsection{The Stechkin problem}

The problem of the best approximation of unbounded operators by linear bounded ones is close to the problem of finding sharp constants in inequalities~(\ref{multiplicative_inequality}),~(\ref{multiplicative_inequality_fractional}) and, furthermore,  presents an independent interest. We follow~\cite{Ste4} (see also surveys~\cite{ArGab,Arestov_96}) to set the problem rigorously.

Let $X$ and $Y$ be the Banach spaces; $A:X\to Y$ be an operator (not necessarily linear) with domain of definition $D_A\subset X$; $Q\subset D_A$ be some set. The function
\begin{equation}
\label{modulus_of_continuity}
	\Omega(\delta) = \Omega(\delta,A,Q) := \sup\left\{\|Af\|_Y\;:\;f\in Q,\; \|f\|_X\leqslant \delta\right\},\qquad \delta\geqslant 0,
\end{equation}
is called {\it the modulus of continuity} of the operator $A$ on the set $Q$.

By $\mathcal{L} = \mathcal{L}(X,Y)$ we denote the space of all linear bounded operators $S:X\to Y$. {\it The error of approximation} of operator $A$ by linear bounded operator $S\in\mathcal{L}$ on set $Q$ is defined by
\[
	U(A,S;Q) := \sup\limits_{x\in Q} \|Ax - Sx\|_Y.
\]
For $N>0$, we set
\begin{equation}
\label{Stechkin_problem}
	E_N(A;Q) := \inf\limits_{S\in \mathcal{L},\;\|S\|\leqslant N} U(A,S;Q).
\end{equation}
The Stechkin problem on the best approximation of operator $A$ by linear bounded operators on set $Q$ consists in evaluating quantity~(\ref{Stechkin_problem}) and finding extremal operators (if any exists) delivering infimum in the right hand part of~(\ref{Stechkin_problem}).

Now, we let
\[
	\ell(\delta,A;Q) := \inf\limits_{N\geqslant 0}\left(E_N(A;Q) + N\delta\right).
\]
The following theorem by S.\,B.~Stechkin~\cite{Ste4} (see also~\cite{BKKP-kn},~\cite{Arestov_96}) provides simple but nevertheless effective lower estimate of quantity~(\ref{Stechkin_problem}) in terms of the modulus of continuity $\Omega$.

{\bf Theorem~A. \it If $A$ is a homogeneous (in particular, linear) operator, $Q\subset D_{A}$ is centrally-symmetric convex set, then for every $N\geqslant 0$ and $\delta\geqslant 0$,
\begin{equation}
\label{Stechkin_estimate}
	E_N(A;Q) \geqslant \sup\limits_{\delta\geqslant 0}\left\{\Omega(\delta,A;Q) - N\delta\right\} = \sup\limits_{x\in Q}\left\{\|Ax\|_Y - N\|x\|_X\right\},
\end{equation}
\[
	\Omega(\delta,A;Q) \leqslant \ell(\delta,A;Q).
\]
Furthermore, if there exists a pair of element $x_0\in Q$ and operator $S_0\in\mathcal{L}$ such that
\begin{equation}
\label{extremal_condition}
	\left\|Ax_0\right\|_Y = U\left(A,S_0;Q\right) + \left\|S_0\right\|\cdot\left\|x_0\right\|_X
\end{equation}
then $\Omega\left(\left\|x_0\right\|_{X},A;Q\right) = \left\|Ax_0\right\|_Y$ and
\[
	E_{\left\|S_0\right\|}\left(A;Q\right) = U\left(A,S_0;Q\right) = \left\|Ax_0\right\|_Y - \left\|S_0\right\|\cdot\left\|x_0\right\|_X.
\]
Consequently, operator $S_0$ is extremal in problem~(\ref{Stechkin_problem}) for $N = \left\|S_0\right\|$, and element $x_0$ -- in problem~(\ref{modulus_of_continuity}) for $\delta = \|x_0\|_X$. }

We refer the reader to the papers~\cite{ArGab,Arestov_96} for the survey of other known results on the Stechkin problem and discussion of related questions.

In particular case when $G = \mathbb{R}$ or $G = \mathbb{R}_+$, $X=L_p(G)$, $Y=L_{\infty}(G)$, $A = D^k_{-}$ and
\[
	Q = W_{p,s}^r(G) := \left\{f\in L_{p,s}^r(G)\;:\;\left\|f^{(r)}\right\|_{L_s(G)}\leqslant 1\right\},
\]
for every $\delta \geqslant 0$, we have
\[
	\Omega\left(\delta, D^k_{-};W^r_{p,s}(G)\right) = K \cdot\delta^{1-\lambda},\qquad \lambda = \frac{k + 1/p}{r-1/s + 1/p},
\]
where $K = \Omega\left(1, D^k_{-};W^r_{p,s}(G)\right)$ is the sharp constant in inequality~(\ref{multiplicative_inequality_fractional}) with $q=\infty$. From the result of V.\,N.~Gabushin~\cite[Lemma~1]{Gab_70} it can be easily shown that estimate~(\ref{Stechkin_estimate}) is sharp when $Y = L_\infty(G)$. Hence, for every $N\geqslant 0$,
\[
	E_N\left(D^{k}_{-};W^r_{p,s}(G)\right) =  \sup\limits_{\delta\geqslant 0}\left\{K\delta^{1-\lambda} - N\delta\right\} = \lambda (1-\lambda)^{\frac{1}{\lambda} - 1}K^{\frac{1}{\lambda}}N^{1-\frac{1}{\lambda}}.
\]
Therefore, in all cases when the sharp constant $K$ in inequality~(\ref{multiplicative_inequality_fractional}) for $q=\infty$ is found, we immediately know the exact value of the quantity of the best approximation of operator $D^{k}_{-}$ by linear bounded operators on the class $W^r_{p,s}(G)$. 

\subsection{The problem of optimal recovery of operators on elements given with an error}

Another problem that is closely related to the Stechkin problem and sharp Kolmogorov type inequalities is the problem of optimal recovery of an operator with the help of the set of linear operators (or mappings in general) on elements of some set that are given with an error. We follow paper~\cite{Arestov_96} to set the problem rigorously.

Let $X$ and $Y$ be the Banach spaces; $A:X\to Y$ be an operator (not necessarily linear) with domain of definition $D_A\subset Q$; $Q\subset D_A$ be some set. By $\mathscr{R}$ we denote either the set $\mathscr{L}$ of all linear operators acting from $X$ to $Y$, or the set of all mappings $\mathscr{O}$ from $X$ to $Y$. For an arbitrary $\delta\geqslant 0$ and $S\in\mathscr{R}$, we set
\[
	U_{\delta}\left(A,S;Q\right) := \sup\left\{\left\|Af - Sf\right\|_{Y}\;:\;f\in Q,\;g\in X,\;\|f-g\|_{X}\leqslant \delta\right\}.
\]
It is clear that $U_{0}\left(A,S;Q\right) = U\left(A,S;Q\right)$. The problem of optimal recovery of the operator $A$ with the help of the set of operators $\mathscr{R}$ on elements of the set $Q$ with given error $\delta$ consists in finding the quantity
\[
	\mathcal{E}_{\delta}\left(\mathscr{R},A;Q\right) := \inf\limits_{S\in\mathscr{R}} U_{\delta}\left(A,S;Q\right),
\]
called the best recovery of operator $A$ with the help of mappings from $\mathscr{R}$ on elements $Q$ given with prescribed error $\delta$. The detailed survey of existing results and further references can be found, for instance, in~\cite{Arestov_96}. The following statement is a corollary of the result by V.\,V.~Arestov~\cite[Theorem~2.1]{Arestov_96} that indicates the close relations between this problem and the Stechkin problem.

{\bf Theorem~B. \it If $A$ is a homogeneous (in particular, linear) operator, $Q\subset D_{A}$ is centrally-symmetric convex set, then for every $N\geqslant 0$ and $\delta\geqslant 0$,
\[
  \Omega(\delta,A;Q) \leqslant \mathcal{E}_{\delta}\left(\mathscr{O},A;Q\right) \leqslant \mathcal{E}_{\delta}\left(\mathscr{L},A;Q\right) \leqslant \ell(\delta,A;Q).
\]
Moreover, if there exist an element $x_0\in Q$ and an operator $S_0\in \mathcal{L}(X,Y)$ satisfying relation~(\ref{extremal_condition}) from Theorem~A then for $\delta = \left\|x_0\right\|_X$,
\[
	\left\|Ax_0\right\|_Y = \Omega(\delta,A;Q) = \mathcal{E}_\delta(\mathscr{O},A;Q) = \mathcal{E}_\delta(\mathscr{L},A;Q).
\]
}

Similarly to the Stechkin problem, in the case $G = \mathbb{R}$ or $G = \mathbb{R}_+$, $X = L_p(G)$ and $Y = L_\infty(G)$, for every $\delta > 0$, we have
\[
	\mathcal{E}_\delta \left(\mathscr{O},D^{k}_{-};W^{r}_{p,s}(G)\right) = \mathcal{E}_\delta \left(\mathscr{L},D^{k}_{-};W^{r}_{p,s}(G)\right) = \Omega\left(\delta,D^{k}_{-};W^{r}_{p,s}(G)\right).
\]
So that once the sharp constant in inequality~(\ref{multiplicative_inequality_fractional}) is found, we immediately know the value of the error of optimal recovery of operator $D^{k}_{-}$ by operators from $\mathscr{O}$ (or $\mathscr{L}$) on elements of the class $W^r_{p,s}(G)$ given with error $\delta$.



\subsection{Organization of the paper}

The paper is organised in the following way. Section~\ref{Sec:aux} is devoted to auxiliary results concerning properties of the Marchaud fractional derivatives: existence, continuity and integral representation in terms of the higher order function derivative. Then, we establish some sufficient conditions when sharp Kolmogorov type inequalities~(\ref{multiplicative_inequality_fractional}) can be written and derive some consequences from these conditions for $r=1,2$ in Section~\ref{Sec:Low_smoothness}. Finally, in Section~\ref{Sec:Applications} we present applications of main results: the Kolmogorov problem for three numbers consisting in finding necessary and sufficient conditions on the triple of real positive numbers that guarantee existence of a function attaining these numbers as the norms of its three consecutive derivatives, and sharp Kolmogorov type inequalities for the weighted norms of the Hadamard fractional derivatives.

\section{Auxiliary Results}
\label{Sec:aux}

In this section we formulate auxiliary propositions on existence and continuity of the Marchaud fractional derivative and its integral representation in terms of higher order derivative. These and similar questions were studied by many mathematicians. For the overview of known results we refer the reader to the books~\cite{Sam,Kil} and references therein.

\subsection{Definitions and results}
\label{Subsec:defs}

Let $G = \mathbb{R}$ or $G = \mathbb{R}_+$. By $\mathfrak{S}(G)$ we denote the space of measurable functions $f:G\to\mathbb{R}$. The linear space $E\subset \mathfrak{S}(G)$ endowed with the norm $\|\cdot\|_{E}$ is called {\it the ideal lattice on $G$} (see~\cite[Ch.~2, \S2]{Krein}) if for every $f\in E$ and  $g\in\mathfrak{S}(G)$ such that $|g(x)| \leqslant |f(x)|$ a. e. on $G$, it follows that $g\in E$ and $\|g\|_{E}\leqslant \|f\|_{E}$. The set $A(E)\subset G$ is called {\it the support} of ideal lattice $E$ if $f(x) = 0$ for every $f\in E$ and $x\not\in A(E)$. By $E^1$ we denote the {\it associated space} to $E$ (see~\cite[Ch.~2, \S3]{Krein}), i.e. the space of functions $g\in \mathfrak{S}(G)$ such that $\textrm{supp}\,g \subset A(E)$ and
\[
	\|g\|_{E^1} := \sup\limits_{f\in E,\atop \|f\|_{E}\leqslant 1} \int_G f(x)g(x)\,dx < \infty.
\]
It is clear that $E^1$ is the ideal lattice on $G$ and is a subspace in the space dual to $E$. Ideal lattices generalize many important spaces e.g. spaces $L_p(G)$, $1\leqslant p\leqslant\infty$, the Orlicz spaces~\cite{Krasn}, the Lorentz spaces~\cite{Krein}, the Marcinkiewicz spaces~\cite{Krein}, etc.

In what follows we would also say that an ideal lattice $E$ is {\it semi shift-invariant} if for every $f\in E$ and $x\in G$ we have $f(\cdot + x)\in E$ and either $\left\|f(\cdot + x)\right\|_{E} = \left\|f\right\|_{E}$ if $G = \mathbb{R}$ or $\left\|f(\cdot + x)\right\|_{E} \leqslant \left\|f\right\|_{E}$ if $G = \mathbb{R}_+$.

Let $r\in\mathbb{N}$, $k\in(0,r)\setminus\mathbb{N}$ and $F$ let be an ideal lattice. By $L_{\infty,E}^{r}(G)$ and $L_{F,E}^r(G)$ we denote the spaces of functions $f\in L_{\infty}(G)$ and $f\in F$ respectively such that $f^{(r-1)}$ is locally absolutely continuous on $G$ and $f^{(r)}\in E$. In addition, let $\chi_B$ stand for the characteristic (indicator) function of a measurable set $B\subset \mathbb{R}$. 

\begin{prop}
\label{Prop:continuityGeneral}
Let $G=\mathbb{R}$ or $G = \mathbb{R}_+$, $r\in\mathbb{N}$, $k\in\left(0,r\right)\setminus\mathbb{N}$, $E$ be a semi shift-invariant lattice on $G$ such that
\begin{equation}
\label{limit_condition0}
	(\cdot)^{r-k-1}\chi_{(0,1)}(\cdot)\in E^1
\end{equation}
and
\begin{equation}
\label{limit_condition}
	\lim\limits_{h\to 0^+} \left\|(\cdot)^{r-1-k}\chi_{(0,h)}(\cdot)\right\|_{E^1} =0
\end{equation}
where $E^1$ is the associated space to $E$. Then $D^{k}_{-}f$ exists and is continuous on $G$ for every function $f\in L^{r}_{\infty,E}(G)$.
\end{prop}

\begin{prop}
\label{Prop:integral_representationGeneral}
Let $G=\mathbb{R}$ or $G = \mathbb{R}_+$, $r\in\mathbb{N}$, $k\in\left(0,r\right)\setminus\mathbb{N}$, $E$ be a semi shift-invariant lattice on $G$ satisfying condition~(\ref{limit_condition0}). Then for every $f\in L^{r}_{\infty,E}(G)$,
\begin{equation}
\label{integral_representationGeneral}
	D^{k}_{-}f(\cdot) = \frac{(-1)^r}{\Gamma(r-k)}\int_0^{+\infty} t^{r-1-k} f^{(r)}(\cdot + t)\,du.
\end{equation}
\end{prop}

\begin{prop}
\label{Prop:existence_and_integral_representation}
Let $G = \mathbb{R}$ or $G = \mathbb{R}_+$, $r\in\mathbb{N}$, $k\in(0,r)\setminus\mathbb{N}$, $E$ be a semi-shift invariant lattice on $G$ satisfying condition~(\ref{limit_condition0}) and $F$ be a semi-shift invariant lattice on $G$ such that $\chi_{(0,1)} \in F^1$. Then $D^k_{-}f(x)$ exists for every $f\in L^r_{F,E}$ and $x\in G$, and integral representation~(\ref{integral_representationGeneral}) for $D^k_{-}f$ holds true.
\end{prop}

In particular, when $E = L_s(G)$, $1\leqslant s\leqslant \infty$, both conditions~(\ref{limit_condition0}) and~(\ref{limit_condition}) are equivalent to the inequality $k < r-1/s$. So that the following corollaries hold true.

\begin{prop}
\label{Prop:continuity}
Let $G=\mathbb{R}$ or $G = \mathbb{R}_+$, $r\in\mathbb{N}$, $1\leqslant p,s\leqslant \infty$ and $k\in\left(0,r-1/s\right)\setminus\mathbb{N}$. Then for every $f\in L_{p,s}^r(G)$, $D^{k}_{-}f$ exists and is continuous on $G$, and~(\ref{integral_representationGeneral}) holds true.
\end{prop}

\begin{prop}
\label{Prop:continuity1}
Let $G=\mathbb{R}$ or $G = \mathbb{R}_+$, $r\in\mathbb{N}$, $1\leqslant s\leqslant \infty$ and $k\in\left(0,r\right)\setminus\mathbb{N}$. Then $D^{k}_{-}f(x)$ exists for every $f\in L_{s,s}^r(G)$ and $x\in G$, and~(\ref{integral_representationGeneral}) holds true.
\end{prop}

\subsection{The proofs of auxiliary results}

For the sake of completeness, we prove Propositions~\ref{Prop:continuityGeneral}--\ref{Prop:existence_and_integral_representation}. Preliminarily, we recall the definition of the B-splines and some of their properties (see, e.g.,~\cite[\S 4.2]{Chui}). The first order B-spline $N_1$ is the function $\chi_{(0,1)}$. For $r\geqslant 2$, the $r$-order B-spline $N_r$ is defined by
\[
	N_r(x) = \int_{\mathbb{R}} N_{r-1}(x-t) N_1(t)\,dt = \int_0^1 N_{r-1}(x-t)\,dt,\qquad x\in\mathbb{R}.
\]
It is known that $N_r$ is continuous and positive on $(0,r)$ function, compactly supported on $[0,r]$. Moreover (see~Theorem~4.3 in~\cite{Chui}), for every $r$-times differentiable function $f:G\to\mathbb{R}$ and every $t>0$,
\begin{equation}
\label{delta-int1}
	\left(\Delta^r_{-t}f\right)\left(x\right)=\left(-1\right)^r t^{r}\int_0^{r}N_r\left(u\right) f^{\left(r\right)}\left(x+ut\right)\,du,\qquad x\in G.
\end{equation}

\begin{proof} [Proof of Proposition~\ref{Prop:continuityGeneral}]
Let a function $f\in L_{\infty,E}^{r}(G)$ and a point $x\in G$ be arbitrary. We observe that  $\left|\left(\Delta^{r}_{-t}f\right)(x)\right| \leqslant 2^{r}\|f\|_{L_{\infty}(G)}$ for every $t>0$. Hence, by definition~(\ref{Marchaud_derivative}) for every $h>0$, we have
\begin{equation}
\label{first_estimate}
	\begin{array}{rcl}
		\varkappa(k,r)\left|D^{k}_{-}f(x)\right| & = & \displaystyle \left|\int_0^{+\infty}\frac{\left(\Delta_{-t}^{r}f\right)(x)}{t^{1+k}}\,dt\right| \\ [10pt]
		& \leqslant & \displaystyle\left|\int_0^{h} \frac{\left(\Delta_{-t}^r f\right)(x)}{t^{1+k}}\,dt\right| + \left|\int_h^{+\infty} \frac{\left(\Delta_{-t}^r f\right)(x)}{t^{1+k}}\,dt\right| \\ [10pt]
		& \leqslant & \displaystyle \left|\int_0^h\frac{\left(\Delta_{-t}^r f\right)(x)}{t^{1+k}}\,dt\right| + \frac{2^r\|f\|_{L_{\infty}(G)}}{k\, h^{k}}.
	\end{array}
\end{equation}
Using formula~(\ref{delta-int1}), changing variables, altering the order of integration and applying the H\"older inequality we obtain
\[
	\begin{array}{rcl}
		\displaystyle \left|\int_0^{h} \frac{\left(\Delta_{-t}^r f\right)(x)}{t^{1+k}}\,dt \right| & = & \displaystyle \left|\int_0^h \int_0^{r} \frac{N_r\left(u\right) f^{(r)}(x+ut)}{t^{k+1-r}}\,du\,dt\right| \\[10pt]
		& = & \displaystyle\left|\int_0^h \int_0^{rt} \frac{N_r\left(v/t\right)f^{(r)}(x+v)}{t^{k+2-r}}\,dv\, dt\right|\\[10pt]
		& = & \displaystyle\left|\int_0^{rh} f^{(r)}(x+v)\int_{v/r}^{h} \frac{N_r\left(v/t\right)}{t^{k+2-r}}\,dt\,dv\right| \\ [10pt]
		& \leqslant & \displaystyle \left\|f^{(r)}\right\|_{E}\cdot \left\|\chi_{(0,rh)}(\cdot)\int_{(\cdot)/r}^h \frac{ N_r\left((\cdot)/t\right)}{t^{k+2-r}}\,dt\right\|_{E^1} .
	\end{array}
\]
It is easy to show that $N_r(x) \leqslant x^{r-1}$ for every $x\in[0,r]$. Hence, for every $v\in(0,rh)$,
\[
	\int_{v/r}^{h} \frac{N_r(v/t)}{t^{k+2-r}}\,dt \leqslant v^{r-1} \int_{v/r}^{h} \frac{dt}{t^{k+1}} \leqslant \frac{v^{r-k-1}}{kr^{k}}.
\]
From the latter and estimate~(\ref{first_estimate}) we conclude that
\[
	\varkappa(k,r)\left|D^{k}_{-}f(x)\right| \leqslant \frac{2^r\|f\|_{L_{\infty}(G)}}{k\,h^{k}} + \frac{\left\|(\cdot)^{r-1-k} \chi_{(0,rh)}(\cdot)\right\|_{E^1}\cdot \left\|f^{(r)}\right\|_{E}}{kr^{k}},
\]
which proves existence and uniform boundedness of derivative $D^{k}_{-}f(x)$ at an arbitrary point $x\in G$.

Now, we turn to the proof of continuity of $D^{k}_{-}f$ on $G$. Let $x\in G$ be an arbitrary point. For every $\varepsilon > 0$, there exist numbers $h>0$ and $H>0$ such that
\[
	\frac{\left\|(\cdot)^{r-k-1}\chi_{(0,rh)}(\cdot)\right\|_{E^1}\cdot \left\|f^{(r)}\right\|_{E}}{kr^{k}\varkappa(k,r)} < \frac{\varepsilon}{6}\qquad\textrm{and}\qquad \frac{2^{r+1}\left\|f\right\|_{L_{\infty}(G)}}{kH^{k}\varkappa(k,r)} < \frac{\varepsilon}{3}.
\]
The function $f$ is continuous on $G$ and is uniformly continuous on $[h/2,rH+rh/2]$. Hence, there exists $\delta\in(0,h/2)$ such that for every $y',y''\in [h/2,rH+rh/2]$, $\left|y' - y''\right| < \delta$, we have $\left|f(y') - f(y'')\right| <2^{-r}  k h^{k} \left|\varkappa(k,r)\right|\varepsilon$. Then using similar arguments as in the proof of existence of $D^{k}_-f$, we obtain that for every $x,y\in G$, $|x - y| < \delta$, for the function $g(t) = f(x+t) - f(y+t)$ we have,
\[
	\begin{array}{l}
		\left|D^{k}_{-}f(x) - D^{k}_{-}f(y)\right| \leqslant \displaystyle \frac{1}{\varkappa(k,r)}\left|\int_{0}^{+\infty} \frac{\left(\Delta_{-t}^{r}g\right)(0)}{t^{1+k}}\,dt\right| \\ [10pt]
		\qquad \leqslant \displaystyle \frac{1}{\varkappa(k,r)}\left(\left|\int_{0}^{h} \frac{\left(\Delta_{-t}^{r}g\right)(0)}{t^{1+k}}\,dt\right| + \left|\int_{h}^{H} \frac{\left(\Delta_{-t}^{r}g\right)(0)}{t^{1+k}}\,dt\right| \right. \\ [10pt]
		\qquad\qquad \displaystyle + \left. \left|\int_{H}^{+\infty} \frac{\left(\Delta_{-t}^{r}g\right)(0)}{t^{1+k}}\,dt\right|\right) \\[10pt]
		\qquad \displaystyle \leqslant \displaystyle\frac{\left\|g^{(r)}\right\|_{E}\cdot\left\|(\cdot)^{r-k-1}\chi_{(0,rh)}(\cdot)\right\|_{E^1}}{kr^{k}\varkappa(k,r)} + \frac{2^{r}\left\|g\right\|_{L_{\infty}([h,rH])}}{kh^{k}\varkappa(k,r)} +  \frac{2^{r}\|g\|_{L_{\infty}(G)}}{kH^{k}\varkappa(k,r)} \\[10pt]
		\qquad\leqslant \displaystyle  \frac{2\left\|f^{(r)}\right\|_{E}\cdot\left\|(\cdot)^{r-k-1}\chi_{(0,rh)}(\cdot)\right\|_{E^1}}{kr^{k}\varkappa(k,r)} + \frac{2^{r}\left\|f(x+\cdot) - f(y+\cdot)\right\|_{L_{\infty}([h,rH])}}{kh^{k}\varkappa(k,r) } \\ [10pt]
		\displaystyle\qquad \qquad +  \frac{2^{r+1}\|f\|_{L_{\infty}(G)}}{kH^{k}\varkappa(k,r) } < \displaystyle \frac{\varepsilon}{3} + \frac{\varepsilon}{3} + \frac{\varepsilon}{3} = \varepsilon.
	\end{array}
\]
Therefore, the $D^{k}_{-}f$ is continuous on $G$.
\end{proof}

{\bf Remark~1.} During the proof of Proposition~\ref{Prop:continuityGeneral} we have established the Kolmogorov type inequality between the uniform norm of fractional derivative of function, the function itself and the norm of its higher order derivative in the ideal lattice:
\[
	\left\|D^{k}_{-}f\right\|_{L_{\infty}(G)} \leqslant \frac{2^r\|f\|_{L_{\infty}(G)}}{k\,h^{k}\varkappa(k,r)} + \left\|\chi_{(0,rh)}(\cdot)\int_{(\cdot)/r}^h \frac{ N_r\left((\cdot)/t\right)}{t^{k+2-r}}\,dt\right\|_{E^1} \cdot \left\|f^{(r)}\right\|_{E}.
\]

\begin{proof}[Proof of Proposition~\ref{Prop:integral_representationGeneral}]
First we note that
\[
	\begin{array}{rcl}
		\displaystyle \int_0^{r} \frac{N_r(u)}{u^{r-k}}\,du & = & \displaystyle \frac{(-1)^{r}}{k(k-1)\ldots (k-r+1)}\,\left(\Delta_{-1}^r (\cdot)^{k}\right)(0)  \\[10pt]
		& = & \displaystyle \frac{(-1)^{r}}{k(k-1)\ldots (k-r+1)}\,\sum\limits_{m=0}^{r} (-1)^{m}\left(r\atop m\right) m^{k}  \\[10pt]
		& = & \displaystyle \frac{\varkappa(k,r)}{\Gamma(-k)(-k)(-k+1)\ldots(-k+r-1)} = \frac{\varkappa(k,r)}{\Gamma(r-k)}.
	\end{array}
\]
Let $f\in L_{\infty,E}^{r}(G)$ and $x\in G$. Derivative $D^{k}_{-}f(x)$ exists due to Proposition~\ref{Prop:continuityGeneral}. Altering the order of integration and applying the Tonelli theorem we obtain
\[
	\begin{array}{rcl}
		\displaystyle D^{k}_{-}f(x) & = & \displaystyle \frac{(-1)^{r}}{\varkappa(k,r)}\int_0^{+\infty}\!\!\!\int_0^{r} t^{r-1-k}N_r\left(u\right) f^{(r)}(x+ut)\,du dt \\[10pt]
		& = & \displaystyle \frac{(-1)^{r}}{\varkappa(k,r)}\int_0^{+\infty} w^{r-k-1} f^{(r)}(x + w) \left(\int_{0}^{r} \frac{N_r\left(u\right)}{u^{r-k}}\,du\right) dw \\ [10pt]
		& = & \displaystyle \frac{(-1)^{r}}{\Gamma(r-k)} \int_0^{+\infty} w^{r-1-k} f^{(r)}(x+w)\,dw
	\end{array}
\]
which finishes the proof.
\end{proof}

\begin{proof}[Proof of Proposition~\ref{Prop:existence_and_integral_representation}]
Let $f\in L^r_{F,E}(G)$ and $x\in G$. Using the same arguments as in the proof of Proposition~\ref{Prop:continuityGeneral} we can prove that the first of two integrals
\[
	\int_0^1 \frac{\left(\Delta_{-t}^r f\right)(x)}{t^{1+k}}\,dt\qquad\textrm{and}\qquad \int_1^{+\infty} \frac{\left(\Delta_{-t}^r f\right)(x)}{t^{1+k}}\,dt
\]
is convergent. Hence, it is sufficient to prove the convergence of the second integral. The latter is obvious because
\[
	\begin{array}{l}
		\displaystyle \left|\int_1^{+\infty} \frac{\left(\Delta_{-t}^r f\right)(x)}{t^{1+k}}\,dt\right| \leqslant \displaystyle \sum\limits_{m=0}^r \left(r\atop m\right) \int_1^{+\infty} \frac{|f(x + mt)|}{t^{1+k}}\,dt \\
		\qquad\qquad = \displaystyle\frac{|f(x)|}{k} + \sum\limits_{m=1}^r m^k\left(r\atop m\right) \int_m^{+\infty} \frac{|f(x + u)|}{u^{1+k}}\,du \\
		\qquad\qquad \leqslant \displaystyle \frac{|f(x)|}{k} + \sum\limits_{m=1}^r m^k\left(r\atop m\right) \|f\|_E\cdot \left\|(\cdot)^{-1-k}\chi_{(m,+\infty)}(\cdot)\right\|_{F^1},
	\end{array}
\]
and for $m=1,2,\ldots,r$,
\[
	\left\|(\cdot)^{-1-k}\chi_{(m,+\infty)}(\cdot)\right\|_{F^1} \leqslant\sum\limits_{j=m}^{\infty} \frac{\left\|\chi_{(0,1)}\right\|_{F^1}}{j^{k+1}} < \frac{m+k}{m^{k+1}k}\, \left\|\chi_{(0,1)}\right\|_{F^1}.
\]
Hence, $D^{k}_{-}f(x)$ exists for every $x\in G$. Finally, we remark that equality~(\ref{integral_representationGeneral}) immediately holds true if $D^{k}_{-}f(x)$ exists. The proof is finished.
\end{proof}

\section{Main Results}
\label{S_new}

Let us present results on some general sufficient conditions allowing to write sharp Kolmogorov type inequality in various situations. We start with the Kolmogorov type inequality between the uniform norms of the Marchaud fractional derivative of a function, the function itself and its higher order derivative. In Subsection~\ref{SubSec:ideal} we extend this result on the case of inequalities between the norms of the function and its derivatives in an ideal lattice. Then in Subsection~\ref{SubSec:mixed} we give another extension of results of Subsection~\ref{SubSec:uniform} on the case of inequalities between the uniform norms of the Marchaud fractional derivative of a function, the uniform norm of the function itself, and the norm of the higher order derivative in an ideal lattice.

Let $G = \mathbb{R}$ or $G = \mathbb{R}_+$. By $V(G)$ we denote the space of functions $f\in L_1(G)$ with bounded on $G$ variation. Also, we set $x_+ := \max\{x;0\}$ for every $x\in G$, and for $f\in L_1(G)$ and $m\in\mathbb{N}$, we denote by $f^{[m]}$ the $m$-th order integral of function $f$:
\[
	f^{[m]}(x) := \frac{1}{(m-1)!} \int_G (x-t)_+^{m-1} f(t)\,dt,\qquad x\in G.
\]
Finally, for $\tau>0$, we define the function $\mathcal{R}_{\tau} : G \to \mathbb{R}$ as follows
\[
	\mathcal{R}_{\tau}(x) = \left\{\begin{array}{ll}
		\frac{x^{\tau-1}}{\Gamma(\tau)}, & x>0,\\ 0, & x\leqslant 0,
	\end{array}\right.
\]

\subsection{The Kolmogorov type inequalities for the Marchaud fractional derivatives: case of uniform norms}
\label{SubSec:uniform}

There holds true the following results.


\begin{thm}
\label{Main_Theorem_uniform}
Let $G = \mathbb{R}_+$ or $G = \mathbb{R}$, $r\in\mathbb{N}$, $k\in(0,r)\setminus\mathbb{N}$, and a function $\Omega\in V(G)$ be such that $\left(\mathcal{R}_{r-k} - \Omega^{[r-1]}\right) \in L_1(G)$ and for every $f\in L^{r}_{\infty,\infty}(G)$,
\begin{equation}
\label{relation}
	D^{k}_{-}f(0) - \int_G f(x) \,d\Omega(x) = (-1)^{r}\int_{G} \left(\mathcal{R}_{r-k}(x) - \Omega^{[r-1]}(x)\right) f^{(r)}(x)\,dx.
\end{equation}
Then for every $f\in L_{\infty,\infty}^{r}(G)$ and $h>0$, there holds true inequality
\begin{equation}
\label{Main_Inequality_uniform}
	\begin{array}{rcl}
		\displaystyle\left\|D^{k}_{-}f\right\|_{L_{\infty}(G)} & \leqslant &\displaystyle h^{-k}\bigvee\limits_{G} \Omega\cdot \|f\|_{L_{\infty}(G)} \\
		& & \qquad \displaystyle+ h^{r-k}\left\|\mathcal{R}_{r-k} - \Omega^{[r-1]}\right\|_{L_1(G)} \cdot\left\|f^{(r)}\right\|_{L_{\infty}(G)}.	
	\end{array}
\end{equation}
Furthermore, if a function $\Phi\in W^{r}_{\infty,\infty}(G)$ satisfies equalities
\begin{equation}
\label{Stiltjes}
	\int_{G} \Phi(t)\, d\Omega(t) = \bigvee_{G}\Omega\cdot \|\Phi\|_{L_{\infty}(G)}
\end{equation}
and
\begin{equation}
\label{Holder_uniform}
	(-1)^{r}\int_G\left(\mathcal{R}_{r-k}(x) - \Omega^{[r-1]}(x)\right)\Phi^{(r)}(x)\,dx = \left\|\mathcal{R}_{r-k} - \Omega^{[r-1]}\right\|_{L_1(G)}
\end{equation}
then~(\ref{Main_Inequality_uniform}) is sharp and the function $\Phi_h(\cdot) := \Phi\left((\cdot)/h\right)$ turns~(\ref{Main_Inequality_uniform}) into equality.
\end{thm}

Minimising the right hand part of~(\ref{Main_Inequality_uniform}) by $h$ we obtain the next consequence.

\begin{cor}
\label{Main_Corollary}
Let $G=\mathbb{R}_+$ or $G = \mathbb{R}$, $r\in\mathbb{N}$, $k\in(0,r)\setminus\mathbb{N}$ and assume that the functions $\Omega$ and $\Phi$ satisfy assumptions of Theorem~\ref{Main_Theorem_uniform}. Then for every $f\in L^{r}_{\infty,\infty}(G)$, there holds true the following sharp inequality
\[
	\left\|D^{k}_{-}f\right\|_{L_{\infty}(G)} \leqslant \frac{\left\|D^{k}_{-}\Phi\right\|_{L_{\infty}(G)}}{\left\|\Phi\right\|_{L_{\infty}(G)}^{1-k/r}} \,\|f\|_{L_{\infty}(G)}^{1-k/r} \left\|f^{(r)}\right\|_{L_{\infty}(G)}^{k/r}.
\]
\end{cor}

We remark that the following results on sharp inequalities of the form~(\ref{multiplicative_inequality_fractional}) concretize Corollary~\ref{Main_Corollary}.
\begin{enumerate}
\item For $G=\mathbb{R}$, $r=2$ and $k\in(0,1)$, extremal function $\Phi$ in inequality~(\ref{Main_Inequality_uniform}) and corresponding function $\Omega$ that satisfy conditions of Corollary~\ref{Main_Corollary} were found by S.\,P.Geiseberg~\cite{geisb} and V.\,V.~Arestov~\cite{arestov} respectively:
\[
	\Omega(x) :=  \left\{\begin{array}{ll}
		0, & x \leqslant 0, \\
		\frac{1}{\Gamma(2-k)}, & x\in(0,1),\\
		\frac{x^{-k}}{\Gamma(1-k)}, & x\geqslant 1,
	\end{array}\right.
	\; \Phi(x) := \left\{\begin{array}{ll}
		-\frac{(1+p)^2}{8}, & x\leqslant -p,\\
		\frac{(x+p)^2-(1+p)^2}{8}, & x\in \left[-p,\frac{1-p}{2}\right],\\
		\frac{(1+p)^2 - (1-x)^2}{8}, & x\in\left[\frac{1-p}{2},1\right],\\
		\frac{(1+p)^2}{8}, & x\geqslant 1,
	\end{array}\right.
\]
where $p = 1-2^{k/(1-k)}$.
\item For $G=\mathbb{R}_+$, $r=2$ and $k\in(0,1)$, extremal function $\Phi$ and corresponding function $\Omega$ that satisfy conditions of Corollary~\ref{Main_Corollary} were found by V.\,V.~Arestov~\cite[Theorem~3]{arestov}:
\[
	\Omega(x) :=  \left\{\begin{array}{ll}
		0, & x = 0, \\
		\frac{1}{\Gamma(2-k)}, & x\in(0,1),\\
		\frac{x^{-k}}{\Gamma(1-k)}, & x\geqslant 1,
	\end{array}\right.
	\qquad\Phi(x) := \left\{\begin{array}{ll}
		\frac 14 - x + \frac{x^2}2, & x\in[0,1],\\
		-\frac 14, & x \geqslant 1.
	\end{array}\right.
\]
\item For $G=\mathbb{R}_+$, $r=2$ and $k\in(1,2)$, extremal function $\Phi$ and corresponding function $\Omega$ that satisfy conditions of Corollary~\ref{Main_Corollary} were also found by V.\,V.~Arestov~\cite[p.~32]{arestov}:
\[
	\begin{array}{rcl}
		\Omega(x) & := & \displaystyle\left\{\begin{array}{ll}
			0, & x = 0, \\
			\frac{3-2^{(k+1)/2}}{\Gamma(2-k)}, & x\in(0,\sqrt{2}-1],\\
			\frac{2^{k/2} - \sqrt{2}}{\Gamma(2-k)\left(\sqrt{2}-1\right)}, & x\in(\sqrt{2}-1,1),\\
			\frac{x^{-k}}{\Gamma(1-k)}, & x\geqslant 1,
		\end{array}\right. \\ [30pt]
		\Phi(x) & := & \displaystyle\left\{\begin{array}{ll}
			\frac{3-2\sqrt{2} - 4\left(\sqrt{2}-1\right)x + 2x^2}{4}, & x\in\left[0,\frac{1}{\sqrt{2}}\right],\\
			\frac{1-2\sqrt{2}+4x - 2x^2}{4}, & x\in\left(\frac 1{\sqrt{2}}, 1\right), \\
			\frac{3-2\sqrt{2}}{4}, & x \geqslant 1.
		\end{array}\right.
	\end{array}
\]
\end{enumerate}

For integer values of $k$, the extremal function $\Phi$ on $\mathbb{R}$ in inequality~(\ref{multiplicative_inequality}) was explicitly constructed by A.\,N.~Kolmogorov~\cite{Kolm39} (see also~\cite{Kolm_85}) for every $r=2,3,\ldots$. We refer the reader to the surveys~\cite{ArGab,Arestov_96,BKKP-kn} for more references and detailed history of the subject and overview of cases when the extremal function $\Phi$ on $\mathbb{R}_+$ is known.

In addition, for integer values of $k$, the function $\Omega$ on $\mathbb{R}_+$ for which inequality~(\ref{Main_Inequality_uniform}) is sharp was explicitly constructed by S.\,B.~Stechkin~\cite{Ste4} in the case $r=2,3$. In case $G=\mathbb{R}$ existence of such function $\Omega$ was proved by Y.~Domar~\cite{Domar_68} and explicitly it was constructed by S.\,B.~Stechkin~\cite{Ste4} for $r=2,3$, V.\,V.~Arestov~\cite{Are_67} for $r=4,5$ and A.\,P.~Buslaev~\cite{Bus_81} for $r>5$.

\begin{proof}[The proof of Theorem~\ref{Main_Theorem_uniform}]
First, we let $h=1$ and define the linear operator $T:L_{\infty}(G) \to L_{\infty}(G)$ as follows
\[
	Tg(\cdot) := \int_G g(\cdot + t)\,d\Omega(t),\qquad g\in L_{\infty}(G).
\]
Clearly, $T$ is bounded and $\|T\| = \bigvee\limits_{G} \Omega$. Next, let a function $f\in L_{\infty,\infty}^{r}(G)$ and a point $x\in G$ be arbitrary. Then from Proposition~\ref{Prop:integral_representationGeneral} and relation~(\ref{relation}) we deduce
\[
	\begin{array}{rcl}
  		\displaystyle \left|D^{k}_{-}f(x)\right| & = & \displaystyle \left|Tf(x) + \left((-1)^r\int_{G} \mathcal{R}_{r-k}(t)f^{(r)}(x + t)\,dt - Tf(x)\right)\right| \\ [10pt]
   		&\leqslant & \displaystyle\left|Tf(x)\right| + \left|(-1)^r\int_{G} \left(\mathcal{R}_{r-k}(t) - \Omega^{[r-1]}(t)\right)f^{(r)}(x+t)\,dt\right| \\ [10pt]
    	& \leqslant & \displaystyle \bigvee\limits_{G}\Omega\cdot\left\|f\right\|_{L_{\infty}(G)} + \left\|\mathcal{R}_{r-k} - \Omega^{[r-1]}\right\|_{L_1(G)}\cdot\left\|f^{(r)}\right\|_{L_{\infty}(G)},
	\end{array}
\]
which implies the desired inequality~(\ref{Main_Inequality_uniform}) in case $h=1$:
\begin{equation}
\label{Main_Inequality_uniform_h=1}
	\left\|D^{k}_{-}f\right\|_{L_{\infty}(G)} \leqslant \bigvee\limits_{G} \Omega\cdot \|f\|_{L_{\infty}(G)} + \left\|\mathcal{R}_{r-k} - \Omega^{[r-1]}\right\|_{L_1(G)} \cdot\left\|f^{(r)}\right\|_{L_{\infty}(G)}.
\end{equation}

Next, we assume that there exists a function $\Phi\in W^{r}_{\infty,\infty}(G)$ satisfying equalities~(\ref{Stiltjes}) and~(\ref{Holder_uniform}). Due to Proposition~\ref{Prop:continuityGeneral} the derivative $D^{k}_{-}\Phi$ is continuous on $G$. Hence, taking into account equalities~(\ref{Stiltjes}) and~(\ref{Holder_uniform}) we have
\[
	\begin{array}{l}
		\displaystyle\left\|D^{k}_{-}\Phi\right\|_{L_{\infty}(G)} \geqslant \displaystyle\left|D^{k}_{-}\Phi(0)\right| \\[10pt]
		\qquad= \displaystyle \left|\int_G \Phi(x)\,d\Omega(x) + (-1)^{r} \int_G \left(\mathcal{R}_{r-k}(x) - \Omega^{[r-1]}(x)\right)\Phi^{(r)}(x)\,dx\right|\\ [10 pt]
		\qquad\geqslant \displaystyle \bigvee_{G}\Omega\cdot \|\Phi\|_{L_{\infty}(G)} + \left\|\mathcal{R}_{r-k} - \Omega^{[r-1]}\right\|_{L_1(G)}\cdot \left\|\Phi^{(r)}\right\|_{L_{\infty}(G)}.
	\end{array}
\]
Therefore, the statement of the theorem is proved in case $h=1$.

Now, we let $h>0$ and $f\in L^r_{\infty,\infty}(G)$ be arbitrary, and consider the function $f_h(x) := f(x/h)$, $x\in G$. Evidently, $f_h \in L^r_{\infty,\infty}(G)$ and by substituting $f_h$ into~(\ref{Main_Inequality_uniform_h=1}) we derive inequality~(\ref{Main_Inequality_uniform}). Clearly, $\Phi_h$ turns~(\ref{Main_Inequality_uniform}) into equality.
\end{proof}

\subsection{The Kolmogorov type inequalities for the Marchaud fractional derivatives: case of norms in an ideal lattice}
\label{SubSec:ideal}

Let us generalize Theorem~\ref{Main_Theorem_uniform} on the case of Kolmogorov type inequalities inequalities between the norms of the Marchaud fractional derivative of a function, the function itself and its higher order derivative in an ideal lattice.

\begin{thm}
\label{Kolmogorov_Stein}
Let $G = \mathbb{R}$ or $G = \mathbb{R}_+$, $E$ be a semi shift-invariant lattice on $G$, $r\in\mathbb{N}$ and $k\in(0,r)\setminus\mathbb{N}$. Let also $\Omega\in V(G)$ be such that $\left(\mathcal{R}_{r-k} - \Omega^{[r-1]}\right)\in L_1(G)$ and relation~(\ref{relation}) hold true for every $f\in L_{E,E}^r(G)$. Then for every $f\in L^r_{E,E}(G)$,
\[
	\left\|D^k_-f\right\|_E \leqslant \bigvee_G\Omega\cdot \|f\|_E + \left\|\mathcal{R}_{r-k} - \Omega^{[r-1]}\right\|_{L_1(G)}\cdot \left\|f^{(r)}\right\|_{E}.
\]
\end{thm}

An immediate consequence of Theorem~\ref{Kolmogorov_Stein} is the following
\begin{cor}
\label{Kolmogorov_Stein_Ls}
Let $G = \mathbb{R}$ or $G = \mathbb{R}_+$, $1\leqslant s\leqslant \infty$, $r\in\mathbb{N}$ and $k\in(0,r)\setminus\mathbb{N}$. Let also a function $\Omega\in V(G)$ be such that $\left(\mathcal{R}_{r-k} - \Omega^{[r-1]}\right)\in L_1(G)$ and relation~(\ref{relation}) holds true for every $f\in L_{s,s}^r(G)$. Then for every $f\in L^r_{s,s}(G)$ and $h>0$,
\[
	\left\|D^k_-f\right\|_{L_s(G)} \leqslant h^{-k}\bigvee_G\Omega\cdot \|f\|_{L_s(G)} + h^{r-k} \left\|\mathcal{R}_{r-k} - \Omega^{[r-1]}\right\|_{L_1(G)}\cdot \left\|f^{(r)}\right\|_{L_s(G)}.
\]
Moreover, if a function $\Phi \in W^r_{\infty,\infty}(G)$ satisfies~(\ref{Stiltjes}) and~(\ref{Holder_uniform}) then for $f\in L_{s,s}^r(G)$,
\begin{equation}
\label{Stein_inequality_Ls}
	\left\|D^k_-f\right\|_{L_s(G)} \leqslant \frac{\left\|D^k_-\Phi\right\|_{L_{\infty}(G)}}{\left\|\Phi\right\|_{L_{\infty}(G)}^{1-k/r}}\,\|f\|_{L_s(G)} ^{1-k/r}\left\|f^{(r)}\right\|_{L_s(G)}^{k/r}.
\end{equation}
\end{cor}

Evidently, inequality~(\ref{Stein_inequality_Ls}) is sharp for $s=\infty$. In Subsection~\ref{SubSec:r=1} we shall show that this inequality is also sharp when $s=1$, $r=1$ and $G = \mathbb{R}$. For integer values of $k$ and $G=\mathbb{R}$, inequality~(\ref{Stein_inequality_Ls}) is known as the Stein inequality~\cite{Stein} (see also~\cite{Bab_Pich_91,KorLigBab}).

\begin{proof}[The proof of Theorem~\ref{Kolmogorov_Stein}]
Using Proposition~\ref{Prop:existence_and_integral_representation} and the generalized Minkowskii inequality~(see~\cite{Krein}) for every function $f\in L^r_{E,E}(G)$, we have
\[
	\begin{array}{rcl}
		\displaystyle \left\|D^k_-f\right\|_E & \leqslant & \displaystyle\left\|\int_G f(x)\,d\Omega(x)\right\|_E + \left\|\int_G \left(\mathcal{R}_{r-k}(x) - \Omega^{[r-1]}(x)\right) f^{(r)}(x)\,dx\right\|_E \\
		& \leqslant & \displaystyle \bigvee_G\Omega\cdot \|f\|_E + \left\|\mathcal{R}_{r-k} - \Omega^{[r-1]}\right\|_{L_1(G)}\cdot \left\|f^{(r)}\right\|_{E}.
	\end{array}
\]
The proof is finished.
\end{proof}

\subsection{The Kolmogorov type inequalities for the Marchaud fractional derivatives: case when the norm of the higher order derivative is considered in an ideal lattice}
\label{SubSec:mixed}


In this Subsection we generalize the results of Subsection~\ref{SubSec:uniform} on the case when the norm of the higher order derivative is taken in an ideal lattice. For convenience, we split the subsection into two parts: first we present results concerning the case when extremal function in the Kolmogorov type inequality (i.e. turning it into equality) exists and then we present results concerning the case when extremal function in the Kolmogorov type inequality does not exist. For the discussion of existence of extremal function in inequality~(\ref{multiplicative_inequality}) for integer order derivatives we refer the reader to the paper~\cite{Bus_Mag_Tih_82} and references therein.


\subsubsection{Case of existence of extremal function in the Kolmogorov type inequality}

For an ideal lattice $E$ on $G$ and $r\in\mathbb{N}$, we set
\[
	W^{r}_{\infty,E}(G) = \left\{f\in L_{\infty,E}^{r}(G):\left\|f^{(r)}\right\|_{E}\leqslant 1 \right\}.
\]

\begin{thm}
\label{Main_Theorem}
Let $G = \mathbb{R}_+$ or $G = \mathbb{R}$, $r\in\mathbb{N}$, $k\in(0,r)\setminus\mathbb{N}$, $E$ be an ideal semi shift-invariant lattice on $G$ satisfying conditions~(\ref{limit_condition0}) and~(\ref{limit_condition}), $E^1$ be the associated space to $E$. Also, let a function $\Omega\in V(G)$ be such that $\left(\mathcal{R}_{r-k} - \Omega^{[r-1]}\right) \in E^1$ and relation~(\ref{relation}) hold true for every $f\in L^{r}_{\infty,E}(G)$. Then for every $f\in L_{\infty,E}^{r}(G)$,
\begin{equation}
\label{Main_Inequality}
	\left\|D^{k}_{-}f\right\|_{L_{\infty}(G)} \leqslant \bigvee\limits_{G} \Omega\cdot \|f\|_{L_{\infty}(G)} + \left\|\mathcal{R}_{r-k} - \Omega^{[r-1]}\right\|_{E^1} \cdot\left\|f^{(r)}\right\|_{E}.
\end{equation}
Furthermore, if a function $\Phi\in W^{r}_{\infty,E}(G)$ satisfies equalities~(\ref{Stiltjes}) and
\begin{equation}
\label{Holder}
	(-1)^{r}\int_G\left(\mathcal{R}_{r-k}(x) - \Omega^{[r-1]}(x)\right)\Phi^{(r)}(x)\,dx = \left\|\mathcal{R}_{r-k} - \Omega^{[r-1]}\right\|_{E^1}
\end{equation}
then inequality~(\ref{Main_Inequality}) is sharp and $\Phi$ turns~(\ref{Main_Inequality}) into equality.
\end{thm}

We remark that Theorem~\ref{Main_Theorem} can be generalized as follows.

\begin{thm}
\label{Theorem_infty_F_E}
Let $G = \mathbb{R}$ or $G = \mathbb{R}_+$, $r\in\mathbb{N}$, $k\in(0,r)\setminus\mathbb{N}$, $E$ be a semi-shift invariant lattice on $G$ that satisfy conditions~(\ref{limit_condition}), $E^1$ be the associated space to $E$, $F$ be an ideal lattice such that its associated space $F^1$ contains the function $\chi_{(0,1)}$. Let also a locally absolutely-continuous on $G$ function $\Omega\in V(G)$ be such that $\left(\mathcal{R}_{r-k} - \Omega^{[r-1]}\right) \in E^1$ and relation~(\ref{relation}) holds true for every $f\in L^{r}_{F,E}(G)$. Then for every $f\in L^r_{F,E}(G)$,
\[
	\left\|D^{k}_{-}f\right\|_{L_{\infty}(G)} \leqslant \left\|\Omega'\right\|_{F^1} \cdot\|f\|_{F} + \left\|\mathcal{R}_{r-k} - \Omega^{[r-1]}\right\|_{E^1} \cdot \left\|f^{(r)}\right\|_{E}.
\]
\end{thm}

For the spaces $L_s(G)$, $1 < s \leqslant \infty$, we obtain the following consequence.

\begin{cor}
\label{Main_Corollary_Ls}
Let $G = \mathbb{R}$ or $G = \mathbb{R}_+$, $1 < s \leqslant \infty$, $s' = s/(s-1)$, $r\in\mathbb{N}$ and $k\in\left(0,r-1/s\right)\setminus\mathbb{N}$. Let also a function $\Omega\in V(G)$ be such that $\left(\mathcal{R}_{r-k} - \Omega^{[r-1]}\right) \in L_{s'}(G)$ and relation~(\ref{relation}) holds true for every $f\in L^r_{\infty,s}(G)$. If a function $\Phi\in W^r_{\infty,s}(G)$ satisfies equality~(\ref{Stiltjes}) and relation
\[
  (-1)^{r}\int_G\left(\mathcal{R}_{r-k}(x) - \Omega^{[r-1]}(x)\right)\Phi^{(r)}(x)\,dx = \left\|\mathcal{R}_{r-k} - \Omega^{[r-1]}\right\|_{L_{s'}(G)}
\]
then for every $f\in L_{\infty,s}^r(G)$ and $h>0$, there hold true sharp inequalities
\begin{equation}
\label{Kolmogorov_inequality_Ls_additive}
	\begin{array}{l}
		\displaystyle\left\|D^{k}_{-}f\right\|_{L_{\infty}(G)} \leqslant h^{-k}\bigvee\limits_{G} \Omega\cdot \|f\|_{L_{\infty}(G)}\\
		\qquad\qquad\qquad\qquad+\displaystyle h^{r-k-1/s}\left\|\mathcal{R}_{r-k} - \Omega^{[r-1]}\right\|_{L_{s'}(G)} \cdot\left\|f^{(r)}\right\|_{L_s(G)},
	\end{array}
\end{equation}
and
\begin{equation}
\label{Kolmogorov_inequality_Ls}
	\left\|D^{k}_{-}f\right\|_{L_{\infty}(G)} \leqslant \frac{\left\|D^{k}_{-}\Phi\right\|_{L_{\infty}(G)}}{\left\|\Phi\right\|_{L_{\infty}(G)}^{1-\lambda}}\,\|f\|_{L_{\infty}(G)}^{1-\lambda} \left\|f^{(r)}\right\|_{L_{s}(G)}^{\lambda},\qquad \lambda = \frac{k}{r-1/s}.
\end{equation}
Moreover, the function $\Phi_h(\cdot) := h^{r-1/s}\Phi\left((\cdot)/h\right)$ turns~(\ref{Kolmogorov_inequality_Ls_additive}) and~(\ref{Kolmogorov_inequality_Ls}) into equalities.
\end{cor}

We remark that Theorems~3.1.2 and 3.2.2~\cite{MBDK} are concretization of Corollary~\ref{Main_Corollary_Ls}. 

\begin{proof}[The proof of Theorem~\ref{Main_Theorem}]
The proof is similar to the proof of Theorem~\ref{Main_Theorem_uniform} in case $h=1$. The  difference is that for a function $f\in L^r_{\infty,E}(G)$ and $x\in G$, we need to use inequality
\[
	\begin{array}{l}
		\displaystyle \left|\int_{G} \left(\mathcal{R}_{r-k}(t) - \Omega^{[r-1]}(t)\right)f^{(r)}(x+t)\,dt\right|\\ \qquad\quad\leqslant \displaystyle\left\|\mathcal{R}_{r-k} - \Omega^{[r-1]}\right\|_{E^1}\cdot\left\|f^{(r)}(x+\cdot)\right\|_{E} \leqslant \displaystyle\left\|\mathcal{R}_{r-k} - \Omega^{[r-1]}\right\|_{E^1}\cdot\left\|f^{(r)}\right\|_{E}.
	\end{array}
\]
The extremity of the function $\Phi$ can be proved in a similar way to Theorem~\ref{Main_Theorem_uniform}.
\end{proof}

\begin{proof}[The proof of Corollary~\ref{Main_Corollary_Ls}]
For every $h>0$, we observe that the functions $\Omega_h(x) := h^{-k}\Omega\left(x/h\right)$, $x\in G$, and $\Phi_h$ satisfy conditions~(\ref{relation}),~(\ref{Stiltjes}) and~(\ref{Holder}). Moreover,
\[
	\bigvee_G \Omega_h = h^{-k}\bigvee_G \Omega,\quad \left\|\mathcal{R}_{r-k} - \Omega^{[r-1]}_h\right\|_{L_{s'}(G)} = h^{r-k-1/s}\left\|\mathcal{R}_{r-k} - \Omega^{[r-1]}\right\|_{L_{s'}(G)}.
\]
Hence, by Theorem~\ref{Main_Theorem} there holds true the desired inequality~(\ref{Kolmogorov_inequality_Ls_additive}) and the function $\Phi_h$ turns~(\ref{Kolmogorov_inequality_Ls_additive}) into equality. Finally, minimizing the right hand part of~(\ref{Kolmogorov_inequality_Ls_additive}) by variable $h$, we arrive at inequality~(\ref{Kolmogorov_inequality_Ls}). The proof is finished.
\end{proof}


\subsubsection{Case of non-existence of extremal function in the Kolmogorov type inequality}

Let us present two results when conditions~(\ref{Stiltjes}) and~(\ref{Holder}) can be relaxed.

\begin{thm}
\label{Main_Theorem1}
Let $G = \mathbb{R}_+$ or $G = \mathbb{R}$, and numbers $k,r$, an ideal semi shift-invariant lattice $E$ on $G$ and a function $\Omega\in V(G)$ satisfy assumptions of Theorem~\ref{Main_Theorem}. Also, assume that there is a family of functions $\left\{\Phi_\varepsilon\right\}_{\varepsilon > 0}\subset W^{r}_{\infty,E}(G)$ satisfying equality~(\ref{Stiltjes}) and for $\varepsilon>0$, inequality
\begin{equation}
\label{Holder1}
	(-1)^{r}\int_G\left(\mathcal{R}_{r-k}(x) - \Omega^{[r-1]}(x)\right)\Phi_\varepsilon^{(r)}(x)\,dx > \left\|\mathcal{R}_{r-k} - \Omega^{[r-1]}\right\|_{E^1} - \varepsilon.
\end{equation}
Then inequality~(\ref{Main_Inequality}) holds true and is sharp in the sense that for every sufficiently small $\varepsilon>0$, there exists a function $f_{\varepsilon} \in L^r_{\infty,E}(G)$ such that
\[
	\left\|D^{k}_{-}f_\varepsilon\right\|_{L_{\infty}(G)} > \bigvee\limits_G \Omega\cdot \left\|f_\varepsilon\right\|_{L_\infty(G)} + \left(\left\|\mathcal{R}_{r-k} - \Omega^{[r-1]}\right\|_{E^1} - \varepsilon\right)\cdot \left\|f^{(r)}_\varepsilon\right\|_E.
\]
\end{thm}

\begin{thm}
\label{Extremal_absent}
Let $G = \mathbb{R}_+$ or $G = \mathbb{R}$, $r\in\mathbb{N}$, $k\in(0,r-1)\setminus\mathbb{N}$, $E$ be an ideal semi shift-invariant lattice on $G$ such that $\liminf\limits_{h\to 0^+} \left(h^{-1}\left\|\chi_{(0,h)}\right\|_E\right) =: \mu\in(0,+\infty)$. Also, let a function $\Omega\in V(G)$ be such that $\left(\mathcal{R}_{r-k} - \Omega^{[r-1]}\right)\in E^1$ and relation~(\ref{relation}) holds true for every $f\in L^r_{\infty,E}(G)$. Assume that there exists a function $\Phi\in L_{\infty}(G)$ such that its derivative $\Phi^{(r-1)}$ is piecewise constant on $G$, $\bigvee\limits_{G}\Phi^{(r-1)} = \mu^{-1}$, there exists $h_0>0$ such that the distance between each pair of discontinuity points of $\Phi^{(r-1)}$ is bounded below by $h_0$, and equalities~(\ref{Stiltjes}) and
\[
	(-1)^r \int_G \left(\mathcal{R}_{r-k}(x) - \Omega^{[r-1]}(x)\right)\,d\Phi^{(r-1)}(x) = \left\|\mathcal{R}_{r-k} - \Omega^{[r-1]}\right\|_{E^1}
\]
are valid. Then inequality~(\ref{Main_Inequality}) holds true and is sharp.
\end{thm}

In case $E = L_1(G)$ we can obtain the following

\begin{cor}
\label{Main_Corollary_Ls_prime}
Let $G = \mathbb{R}$ or $G = \mathbb{R}_+$, $r\in\mathbb{N}$ and $k\in\left(0,r-1\right)\setminus\mathbb{N}$. Let also a function $\Omega\in V(G)$ be such that $\left(\mathcal{R}_{r-k} - \Omega^{[r-1]}\right) \in L_{\infty}(G)$ and relation~(\ref{relation}) holds true for every $f\in L^r_{\infty,1}(G)$. If an $(r-1)$-times differentiable function $\Phi$ with piecewise constant derivative $\Phi^{(r-1)}$ satisfies equalities~(\ref{Stiltjes}), $\bigvee\limits_G\Phi^{(r-1)} = 1$ and
\[
  (-1)^{r}\int_G\left(\mathcal{R}_{r-k}(x) - \Omega^{[r-1]}(x)\right)\,d\Phi^{(r-1)}(x) = \left\|\mathcal{R}_{r-k} - \Omega^{[r-1]}\right\|_{L_{\infty}(G)}
\]
then for every $f\in L_{\infty,1}^r(G)$ and $h>0$, there holds true sharp inequalities
\[
	\begin{array}{l}
		\displaystyle\left\|D^{k}_{-}f\right\|_{L_{\infty}(G)} \leqslant h^{-k}\bigvee\limits_{G} \Omega\cdot \|f\|_{L_{\infty}(G)}\\
		\displaystyle\qquad\qquad\qquad\qquad+ h^{r-k-1}\left\|\mathcal{R}_{r-k} - \Omega^{[r-1]}\right\|_{L_{\infty}(G)} \cdot\left\|f^{(r)}\right\|_{L_1(G)},
	\end{array}
\]
and
\[
	\left\|D^{k}_{-}f\right\|_{L_{\infty}(G)} \leqslant \frac{\left\|D^{k}_{-}\Phi\right\|_{L_{\infty}(G)}}{\left\|\Phi\right\|_{L_{\infty}(G)}^{1-\lambda}}\,\|f\|_{L_{\infty}(G)}^{1-\lambda} \left\|f^{(r)}\right\|_{L_{1}(G)}^{\lambda},\qquad \lambda = \frac{k}{r-1}.
\]
\end{cor}

\begin{proof}[The proof of Theorem~\ref{Main_Theorem1}]
We observe that inequality~(\ref{Main_Inequality}) holds true for every $f\in L^r_{\infty,E}(G)$. Let us prove that~(\ref{Main_Inequality}) is sharp. Let $\varepsilon>0$ be arbitrary and sufficiently small. Due to Proposition~\ref{Prop:continuityGeneral} the fractional derivative $D^{k}_{-}\Phi_\varepsilon$ is continuous on $G$. Hence, taking into account equalities~(\ref{Stiltjes}) and~(\ref{Holder1}) we obtain
\[
	\begin{array}{l}
		\displaystyle\left\|D^{k}_{-}\Phi_\varepsilon\right\|_{L_{\infty}(G)} \geqslant \displaystyle\left|D^{k}_{-}\Phi_\varepsilon(0)\right| \\
		\displaystyle\qquad= \left|\int_G \Phi_\varepsilon(x)\,d\Omega(x) + (-1)^{r} \int_G \left(\mathcal{R}_{r-k}(x) - \Omega^{[r-1]}(x)\right)\Phi_\varepsilon^{(r)}(x)\,dx\right| \\
		\displaystyle\qquad\geqslant \displaystyle \bigvee_{G}\Omega\cdot \left\|\Phi_\varepsilon\right\|_{L_{\infty}(G)} + \left\|\mathcal{R}_{r-k} - \Omega^{[r-1]}\right\|_{E^1} - \varepsilon \\
		\displaystyle\qquad \geqslant \displaystyle \bigvee_{G}\Omega\cdot \left\|\Phi_\varepsilon\right\|_{L_{\infty}(G)} + \left(\left\|\mathcal{R}_{r-k} - \Omega^{[r-1]}\right\|_{E^1} - \varepsilon\right)\cdot \left\|\Phi_\varepsilon^{(r)}\right\|_{E}.
	\end{array}
\]
The proof is finished.
\end{proof}



\begin{proof}[The proof of Theorem~\ref{Extremal_absent}]
Since $k<r-1$ we see that $(\cdot)^{r-k-1}\chi_{(0,1)}(\cdot)\in E^1$ and condition~(\ref{limit_condition}) is also fulfilled. Hence, by Theorem~\ref{Main_Theorem}, inequality~(\ref{Main_Inequality}) holds true. Let us prove that inequality~(\ref{Main_Inequality}) is sharp. To this end by $B = \left\{\beta_j\right\}_{j\in J}$ ($J$ is a finite or countable set of indices) we denote the discontinuity points of $\Phi^{(r-1)}$ and set $\alpha_j := \Phi\left(\beta_j^+\right) - \Phi\left(\beta_j^{-}\right)$, $j\in J$. Due to assumption there exists $h_0>0$ such that $\left|\beta_j - \beta_i\right| \geqslant h_0$ for every distinct indexes $i,j\in J$. Now, for every $h\in\left(0,h_0\right)$, we define the function
\[
	\Phi_h(x) := \frac 1h \int_0^h \Phi(x + t)\,dt,\qquad x\in G.
\]
It is easy to show that as $h\to 0^+$, we have
\[
	\int_{G} \Phi_h(x)\, d\Omega(x) \to \int_{G} \Phi(x)\, d\Omega(x) = \bigvee_{G}\Omega\cdot \|\Phi\|_{L_{\infty}(G)},
\]
\[
	\begin{array}{l}
		\displaystyle\int_G\left(\mathcal{R}_{r-k}(x) - \Omega^{[r-1]}(x)\right)\Phi_h^{(r)}(x)\,dx \to \displaystyle \int_G\left(\mathcal{R}_{r-k}(x) - \Omega^{[r-1]}(x)\right)\,d\Phi^{(r-1)}(x) \\
		\qquad\qquad\qquad\qquad\quad\qquad \qquad\qquad\qquad=\displaystyle(-1)^{r}\left\|\mathcal{R}_{r-k} - \Omega^{[r-1]}\right\|_{E^1}
	\end{array}
\]
and
\[
	\liminf\limits_{h\to 0^+}\left\|\Phi_h^{(r)}\right\|_{E} \leqslant \liminf\limits_{h\to 0^+}\sum\limits_{j\in J} \frac{\left|\alpha_j\right|}{h}\cdot \left\|\chi_{(0,h)}\right\|_{E} = \frac{1}{\mu}\cdot\liminf\limits_{h\to 0^+}\frac{\left\|\chi_{(0,h)}\right\|_{E}}{h} = 1.
\]
So that due to continuity of $D^{k}_{-}\Phi_h$, for every $\varepsilon$ and every sufficiently small $h>0$,
\[
	\begin{array}{rcl}
		\displaystyle\left\|D^{k}_{-}\Phi_h\right\|_{L_\infty(G)} & \geqslant & \displaystyle D^{k}_{-}\Phi_h(0) > \bigvee_G \Omega\cdot \left\|\Phi_h\right\|_{L_\infty(G)} + \left\|\mathcal{R}_{r-k} - \Omega^{[r-1]}\right\|_{E^1} - \varepsilon \\
		& \geqslant & \displaystyle \bigvee_G \Omega\cdot \left\|\Phi_h\right\|_{L_\infty(G)} + \left(\left\|\mathcal{R}_{r-k} - \Omega^{[r-1]}\right\|_{E^1} - \varepsilon\right) \left\|\Phi_h^{(r)}\right\|_{E}.
	\end{array}
\]
The proof is finished.
\end{proof}

\section{Consequences of main results}
\label{Sec:Low_smoothness}

In this section we deduce new sharp Kolmogorov type inequalities from the results of the previous section when the order of the higher order derivative is $1$ or $2$.

\subsection{Case $r=1$ and $k\in(0,1)$}
\label{SubSec:r=1}

Let $G = \mathbb{R}$ or $G = \mathbb{R}_+$. For $k\in(0,1)$ and $h>0$, we set
\[
	\tau_h(x) := \left\{\begin{array}{ll}
		0, & x\not \in G\setminus(0,h), \\
		x^{-k} - h^{-k}, & x\in (0,h).
	\end{array}\right.
\]
The following proposition is the consequence of Theorem~\ref{Main_Theorem1}.

\begin{cor}
\label{Cor:r=1}
Let $G = \mathbb{R}$ or $G = \mathbb{R}_+$, $k\in(0,1)$, $E$ be an ideal semi shift-invariant lattice on $G$ satisfying conditions~(\ref{limit_condition0}) and~(\ref{limit_condition}), $E^1$ be the associated space to $E$. Then for every $f\in L^1_{\infty,E}(G)$ and $h>0$, there holds true sharp inequality
\begin{equation}
\label{Kolmgorov_inequality_r=1}
	\left\|D^{k}_{-}f\right\|_{L_{\infty}(G)} \leqslant \frac{2h^{-k}}{\Gamma(1-k)}\,\|f\|_{L_{\infty}(G)} + \frac{\left\|\tau_h\right\|_{E^1}}{\Gamma(1-k)}\,\left\|f'\right\|_{E}.
\end{equation}
\end{cor}

\begin{proof}
For every $h>0$, we define the function
\begin{equation}
\label{Omega_r=1}
 \Omega(x) := \mathcal{R}_{1-k}(x) - \frac{\tau_h(x)}{\Gamma(1-k)} = \left\{
 	\begin{array}{ll}
 		0, & x\in G\setminus(0,+\infty), \\
 		\frac{h^{-k}}{\Gamma(1-k)}, & x\in(0,h),\\
 		\frac{x^{-k}}{\Gamma(1-k)}, & x \geqslant h.
 	\end{array}\right.
\end{equation}
It is easy to check that relation~(\ref{relation}) holds true for every $f\in L^{1}_{\infty,E}(G)$, $\bigvee\limits_G \Omega = \frac{2h^{-k}}{\Gamma(1-k)}$ and $\mathcal{R}_{1-k} - \Omega = \frac{\tau_h}{\Gamma(1-k)} \in E^1$.

Let us construct a family of functions $\left\{\Phi_{\varepsilon}\right\}_{\varepsilon>0}\subset W^{1}_{\infty,E}(G)$ satisfying conditions~(\ref{Stiltjes}) and~(\ref{Holder1}). For every $\varepsilon > 0$, there exists a function $g_{\varepsilon}\in E$, $\left\|g_{\varepsilon}\right\|_{E} \leqslant 1$, such that
\[
	\int_0^{h} \left(\mathcal{R}_{1-k}(x) - \Omega(x)\right) g_{\varepsilon}(x)\,dx > \left\|\mathcal{R}_{1-k} - \Omega\right\|_{E^1} - \varepsilon = \frac{\left\|\tau_h\right\|_{E^1}}{\Gamma(1-k)} - \varepsilon.
\]
Without loss of generality we may assume that $g_{\varepsilon}$ is non-negative on $G$ and $\textrm{supp}\,g_{\varepsilon} = [0,h]$. Next, we define the function $\Phi_{\varepsilon}$ as the first integral of $\left(-g_{\varepsilon}\right)$:
\[
	\Phi_{\varepsilon}(x) := -\int_0^{x} g_{\varepsilon}(t)\,dt + \frac{1}{2} \int_0^{h} g_{\varepsilon}(t)\,dt,\qquad x\in G.
\]
Clearly, $\Phi_{\varepsilon} \in W^{1}_{\infty,E}(G)$ and $\Phi_{\varepsilon}(x) = -\Phi_\varepsilon(0) = - \left\|\Phi_{\varepsilon}\right\|_{L_{\infty}(G)}$, $x\geqslant h$. As a result,
\[
	\int_G \Phi_{\varepsilon}(x)\,d\Omega(x) = \frac{h^{-k}\Phi_{\varepsilon}(0)}{\Gamma(1-k)} + \int_h^{\infty} \frac{\Phi_{\varepsilon}(x)\,d\left(x^{-k}\right)}{\Gamma(1-k)} = \bigvee\limits_{G}\Omega\cdot\left\|\Phi_{\varepsilon}\right\|_{L_{\infty}(G)}
\]
and
\[
	\begin{array}{rcl}
		\displaystyle-\int_G \left(\mathcal{R}_{1-k}(x) - \Omega(x)\right)\Phi'_{\varepsilon}(x)\,dx & = &\displaystyle \int_0^{h}\left(\mathcal{R}_{1-k}(x) - \Omega(x)\right) g_{\varepsilon}(x)\,dx \\ [10pt]
		& > &
		\displaystyle\left\|\mathcal{R}_{1-k} - \Omega\right\|_{E^1} - \varepsilon.
	\end{array}
\]
Therefore, the function $\Omega$ and the family of functions $\left\{\Phi_{\varepsilon}\right\}_{\varepsilon>0}$ satisfy assumptions of Theorem~\ref{Main_Theorem1}. Hence, inequality~(\ref{Kolmgorov_inequality_r=1}) holds true and is sharp.
\end{proof}

Next, we formulate the following Stein type inequality.
\begin{cor}
\label{Cor:Stein_r=1}
For $k\in(0,1)$, $h>0$ and $f\in L^1_{\infty,1}(\mathbb{R})$, there hold true sharp inequalities
\begin{equation}
	\begin{array}{rcl}
		\displaystyle \left\|D^k_-f\right\|_{L_{1}(\mathbb{R})} & \leqslant & \displaystyle \frac{2h^{-k}\left\|f\right\|_{L_1(\mathbb{R})}}{\Gamma(1-k)}  + \frac{kh^{1-k}\left\|f'\right\|_{L_1(\mathbb{R})}}{\Gamma(2-k)}, \label{Stein_r=1} \\[8pt]
		\displaystyle \left\|D^k_-f\right\|_{L_{1}(\mathbb{R})} & \leqslant & \displaystyle\frac{2^{1-k}}{\Gamma(2-k)}\, \left\|f\right\|^{1-k}_{L_1(\mathbb{R})}\left\|f'\right\|_{L_1(\mathbb{R})}^{k}.
	\end{array}
\end{equation}
\end{cor}

\begin{proof}
For $h>0$, let $\Omega$ be defined by~(\ref{Omega_r=1}). Then both desired inequalities follow from Corollary~\ref{Kolmogorov_Stein_Ls}. Let us prove that inequalities~(\ref{Stein_r=1}) are sharp. To this end for every $\varepsilon\in(0,h)$, we consider the Steklov averaging operator $S_\varepsilon:L_\infty(G)\to L_\infty(G)$
\[
	S_{\varepsilon} f(\cdot) = \frac{1}{\varepsilon}\int_{0}^{\varepsilon} f(\cdot + t)\,dt,\qquad f\in L_\infty(G),
\]
and define the function $\Phi_{\varepsilon} := S_{\varepsilon}\chi_{(0,h)}$. Clearly, $\left\|\Phi_{\varepsilon}\right\|_{L_1(\mathbb{R})} = h$, $\left\|\Phi'_{\varepsilon}\right\|_{L_1(\mathbb{R})} = 2$ and
\[
	\lim\limits_{\varepsilon \to 0^+} \left\|D^{k}_-\Phi_\varepsilon\right\|_{L_{1}(\mathbb{R})} = \left\|D^{k}_-\chi_{(0,h)}\right\|_{L_{1}(\mathbb{R})} =  \frac{2h^{1-k}}{\Gamma(2-k)}.
\]
Plugging the latter relations into the first of inequalities~(\ref{Stein_r=1}) we turn it into equality. The proof is finished. 
\end{proof}

\subsection{Case $G = \mathbb{R}_+$, $r=2$ and $k\in(0,1)$}
\label{SubSec:R+k0..1r=2}

For $k\in(0,1)$ and $h>0$, we define
\[
	\tau_h(x) := \left\{\begin{array}{ll}
		x^{1-k} - h^{-k}x, & x\in[0,h), \\
		0, & x \geqslant h.
	\end{array}\right.
\]
There holds true the following consequence of Theorem~\ref{Main_Theorem1}.

\begin{cor}
\label{Cor:r=2}
Let $k\in(0,1)$, $E$ be an ideal semi shift-invariant lattice on $\mathbb{R}_+$ satisfying conditions~(\ref{limit_condition0}) and~(\ref{limit_condition}), $E^1$ be the associated space to $E$. Then for every $f\in L^2_{\infty,E}\left(\mathbb{R}_+\right)$ and $h>0$, there holds true sharp inequality
\begin{equation}
\label{Kolmogorov_inequality_r=2}
  \left\|D^{k}_{-}f\right\|_{L_{\infty}\left(\mathbb{R}_+\right)} \leqslant \frac{2h^{-k}}{\Gamma(2-k)}\,\|f\|_{L_{\infty}\left(\mathbb{R}_+\right)} + \frac{\left\|\tau_h\right\|_{E^1}}{\Gamma(2-k)}\,\left\|f''\right\|_{E}.
\end{equation}
\end{cor}

\begin{proof}
For every $h>0$, we define the function
\[
	\Omega(x) := \left\{
	\begin{array}{ll}
		0, & x=0, \\
		\frac{h^{-k}}{\Gamma(2-k)}, & x\in(0,h),\\
		\frac{x^{-k}}{\Gamma(1-k)}, & x \geqslant h.
	\end{array}\right.
\]
Evidently, $\bigvee\limits_G \Omega = \frac{2h^{-k}}{\Gamma(2-k)}$ and $\mathcal{R}_{2-k} - \Omega^{[1]} = \frac{\tau_h}{\Gamma(2-k)}\in E^1$. Moreover, for every $f\in L^2_{\infty,E}\left(\mathbb{R}_+\right)$, there holds true relation~(\ref{relation}). Indeed,
\[
	\begin{array}{l}
		\displaystyle D^k_-f(0) - \int_0^{+\infty} f(x)\, d\Omega(x) \\
			\qquad= \displaystyle D^k_-f(0) - \frac{h^{-k}f(0)}{\Gamma(2-k)} + \frac{kh^{-k}f(h)}{\Gamma(2-k)} + \frac{k}{\Gamma(1-k)}\int_h^{\infty} \frac{f(x)\,dx}{x^{1+k}} \\
			\qquad= \displaystyle \frac{k}{\Gamma(1-k)}\int_0^{h} \frac{f(0) - f(x)}{x^{1+k}}\,dx +  \frac{kh^{-k}f(h)}{\Gamma(2-k)} \\
			\qquad=\displaystyle - \frac{k}{\Gamma(1-k)} \int_0^{h} \int_0^x \frac{(x-t) f''(t)}{x^{1+k}}\, dt\,dx  +  \frac{kh^{-k}}{\Gamma(2-k)}\int_0^h (h-t)f''(t)\,dt\\
			\qquad=\displaystyle\int_0^{+\infty} \left(\mathcal{R}_{2-k}(t) - \Omega^{[1]}(t)\right) f''(t)\,dt.
	\end{array}
\]

Let us construct a family of functions $\left\{\Phi_{\varepsilon}\right\}_{\varepsilon>0}\subset W^{2}_{\infty,E}(G)$ satisfying conditions~(\ref{Stiltjes}) and~(\ref{Holder1}). For every $\varepsilon > 0$, there exists a function $g_{\varepsilon}\in E$, $\left\|g_{\varepsilon}\right\|_{E} \leqslant 1$, such that
\[
	\int_0^{h} \left(\mathcal{R}_{2-k}(x) - \Omega^{[1]}(x)\right) g_{\varepsilon}(x)\,dx > \left\|\mathcal{R}_{2-k} - \Omega^{[1]}\right\|_{E^1} - \varepsilon = \frac{\left\|\tau_h\right\|_{E^1}}{\Gamma(2-k)} - \varepsilon.
\]
Without loss of generality we may assume that $g_{\varepsilon}$ is non-negative on $G$ and $\textrm{supp}\,g_{\varepsilon} = [0,h]$. Next, we define the function $\Phi_{\varepsilon}$ as the second integral of $g_{\varepsilon}$:
\[
	\Phi_{\varepsilon}(x) := \int_0^h \left( - x + t/2\right) g_{\varepsilon}(t)\,dt + \int_0^x (x - t) g_{\varepsilon}(t)\,dt,\qquad x\in\mathbb{R}_+.
\]
Clearly, $\Phi_{\varepsilon} \in W^{2}_{\infty,E}(G)$, $\Phi_\varepsilon(x) = -\Phi_{\varepsilon}(0) = -\left\|\Phi_{\varepsilon}\right\|_{L_\infty\left(\mathbb{R}_+\right)}$, $x\geqslant h$. As a result,
\[
	\begin{array}{rcl}
		\displaystyle\int_0^{+\infty} \Phi_{\varepsilon}(x) \,d\Omega(x) & = & \displaystyle \frac{h^{-k}\left[\Phi_{\varepsilon}(0) - k\Phi_{\varepsilon}(h)\right]}{\Gamma(2-k)} - \frac{k(1-k)}{\Gamma(2-k)}\int_h^{+\infty} \frac{\Phi_{\varepsilon}(x)}{x^{1+k}}\,dx\\[10pt]
		& =& \displaystyle \bigvee_{0}^{+\infty} \Omega\cdot \left\|\Phi_{\varepsilon}\right\|_{L_\infty\left(\mathbb{R}_+\right)}
	\end{array}
\]
and
\[
	\begin{array}{rcl}
		\displaystyle\int_0^{+\infty} \left(\mathcal{R}_{2-k}(x) - \Omega^{[1]}(x)\right)\Phi''_{\varepsilon}(x)\,dx & = & \displaystyle\int_0^h \left(\mathcal{R}_{2-k}(x) - \Omega^{[1]}(x)\right) g_{\varepsilon}(x)\,dx\\ [10pt]
		& > & \displaystyle\left\|\mathcal{R}_{2-k} - \Omega^{[1]}\right\|_{E^1} - \varepsilon.
	\end{array}
\]
Therefore, the function $\Omega$ and the family of functions $\left\{\Phi_{\varepsilon}\right\}_{\varepsilon>0}$ satisfy assumptions of Theorem~\ref{Main_Theorem1}. Hence, inequality~(\ref{Kolmogorov_inequality_r=2}) holds true and is sharp.
\end{proof}

Let us formulate the consequence of Corollaries~\ref{Main_Corollary_Ls},~\ref{Main_Corollary_Ls_prime} and~\ref{Cor:r=2}. For $s>1$, we set
\[
	\varphi_{k,s}(x) := \int_0^h \left(- x + t/2\right) \tau_1^{s'-1}(t)\,dt + \int_0^{x} (x - t) \tau_1^{s'-1}(t)\,dt,\qquad x\in\mathbb{R}_+,
\]
and $\Phi_{k,s} := \left\|\varphi_{k,s}\right\|_{L_s\left(\mathbb{R}_+\right)}^{-1}\cdot \varphi_{k,s}$. Also, we define
\[
	\Phi_{k,1}(x) = \frac{1}{2}\cdot\max{\left\{(1-k)^{1/k} - 2x;\,-(1-k)^{1/k}\right\}},\qquad x\in\mathbb{R}_+.
\]


\begin{cor}
\label{Cor:Main_Corollary_Ls}
Let $k\in(0,1)$, $1\leqslant s\leqslant\infty$ and $s'=s/(s-1)$. Then for every $f\in L_{\infty,s}^2\left(\mathbb{R}_+\right)$, there holds true sharp inequality
\[
  \left\|D^{k}_{-}f\right\|_{L_{\infty}\left(\mathbb{R}_+\right)} \leqslant \frac{\left\|D^{k}_{-}\Phi_{k,s}\right\|_{L_{\infty}\left(\mathbb{R}_+\right)}}{\left\|\Phi_{k,s}\right\|_{L_\infty\left(\mathbb{R}_+\right)}^{1-\lambda}}\left\|f\right\|_{L_{\infty}\left(\mathbb{R}_+\right)}^{1-\lambda} \left\|f''\right\|^{\lambda}_{L_s\left(\mathbb{R}_+\right)},\quad \lambda = \frac{k}{2-1/s}.
\]
\end{cor}

\subsection{Case $G = \mathbb{R}_+$, $1< s\leqslant\infty$ and $k\in\left(1, 2-1/s\right)$}
\label{SubSec:R+k1..2r=2}

Let $1< s\leqslant\infty$, $s' = s/(s-1)$ and $k\in\left(1,2-1/s\right)$. Consider the set
\[
	M := \left\{(a,b)\in(0,1)^2\,:\,a\leqslant b\right\}
\]
and for every $(a,b)\in M$, we define the function
\[
	\omega(a,b;x) = \left\{\begin{array}{ll}
		0, & x=0, \\
		\displaystyle a^{-1}(1-b)^{-1}\cdot \left(1-b-\left(1-b^{1-k}\right)(1-a)\right), & x\in(0,a], \\
		\displaystyle (1-b)^{-1}\cdot\left(1-b^{1-k}\right), & x\in (a,1),\\
		\displaystyle (1-k)x^{-k}, & x\geqslant 1.
	\end{array}\right.
\]
For $x\in\mathbb{R}_+$, we consider functions $\tau(a,b;x) := \Gamma(2-k)\cdot \mathcal{R}_{2-k}(x) - \omega^{[1]}(a,b;x)$ and
\[
	\displaystyle \varphi(a,b;x) := \displaystyle \int_0^{a}\left(-x+t/2\right)\cdot\tau_{(s')}(a,b;t)\,dt + \int_{0}^{x}(x-t)\cdot\tau_{(s')}(a,b;t)\,dt
\]
where $g_{(s')} := |g|^{s'-1}\textrm{sign}\,g$. Below in Lemma~\ref{Lem:Existence_R+_k_1..2} we shall show that the system~(\ref{2nd_equation}) has a unique solution $\left(a_{k,s},b_{k,s}\right)$ on $M$. For convenience, we denote the functions $\omega\left(a_{k,s},b_{k,s};\cdot\right)$, $\tau\left(a_{k,s},b_{k,s};\cdot\right)$, $\varphi\left(a_{k,s},b_{k,s};\cdot\right)$ by $\omega_{k,s}$, $\tau_{k,s}$, $\varphi_{k,s}$ respectively. The graphs of functions $\omega^{[1]}(a,b;\cdot)$, $\tau(a,b;\cdot)$ and $\varphi_{k,s}$ are shown in Figure~\ref{fig:1.3} on page~\pageref{fig:1.3}.

The next proposition is the consequence of Corollary~\ref{Main_Corollary_Ls}.

\begin{cor}
\label{Cor:Kolmogorov_inequality_R+_k_1..2}
Let $1<s\leqslant\infty$, $s' = s/(s-1)$, $k\in\left(1,2-1/s\right)$ and $\Phi_{k,s} := \left\|\varphi_{k,s}\right\|_{L_s\left(\mathbb{R}_+\right)}^{-1}\cdot \varphi_{k,s}$. Then for every $f\in L_{\infty,s}^{2}\left(\mathbb{R}_+\right)$, there hold true sharp inequality 
\[
  \left\|D^{k}_{-}f\right\|_{L_\infty\left(\mathbb{R}_+\right)} \leqslant \frac{\left\|D^{k}_{-}\Phi_{k,s}\right\|_{L_\infty\left(\mathbb{R}_+\right)}}{\left\|\Phi_{k,s}\right\|_{L_\infty\left(\mathbb{R}_+\right)}^{1-\lambda}} \left\|f\right\|_{L_\infty\left(\mathbb{R}_+\right)}^{1-\lambda} \left\|f''\right\|_{L_s\left(\mathbb{R}_+\right)}^{\lambda},\quad \lambda = \frac{k}{2-1/s},
\]
\end{cor}


We start with the proof of auxiliary

\begin{lem}
\label{Lem:Existence_R+_k_1..2}
The system of the following equations has a unique solution on $M$:
\begin{equation}
\label{2nd_equation}
	\left\{\begin{array}{ll}
		\displaystyle F_1(a,b) := \int_{a}^{1} \tau_{(s')}(a,b;t)\,dt = 0, \\ [8pt]
		\displaystyle F_2(a,b) := \int_0^1 t\cdot \tau_{(s')}(a,b;t)\,dt = 0.
	\end{array}\right.
\end{equation}
\end{lem}

We remark that in some cases the pair $\left(a_{k,s},b_{k,s}\right)$ can be found explicitly, e. g. $a_{k,\infty} = \sqrt{2}-1$ and $b_{k,\infty} = 1/\sqrt{2}$.

\begin{proof} First, we observe that for every $(a,b)\in M$, the function $\tau(a,b;\cdot)$ is positive on $(0,b)$, is negative on $(b,1)$, and $\textrm{supp}\,\tau(a,b;\cdot) = [0,1]$. Next, the functions $F_1$ and $F_2$ are continuous on $M$, and can be continuously extended on a wider set $\widetilde{M} := \left\{(a,b)\in[0,1)^2\,:\,b>0,\,a\leqslant b\right\}$. Let us prove that system~(\ref{2nd_equation}) has a unique solution on $\widetilde{M}$. First, we note that $F_2$ strictly increases in variables $a$ and $b$, while $F_1$ strictly increases in $b$ and strictly decreases in $a$. Hence, the system~(\ref{2nd_equation}) could have only one solution. Next, we have
\[
	\begin{array}{rcl}
		\displaystyle\lim\limits_{b\to 0^+} F_2(b,b) & \leqslant & \displaystyle\lim\limits_{b\to 0^+} \left\{\int_0^{b} t^{(1-k)(s'-1) + 1}\,dt - \int_{b}^{1} t\cdot \tau_{(s')}(b,b;t)\,dt \right\} = -\infty, \\[8pt]
		\displaystyle\lim\limits_{b\to 1^-}F_2(b,b) & = & \displaystyle\int_0^{1} t \left(t^{1-k}-t\right)^{s'-1}\,dt > 0,\\[8pt]
		\displaystyle\lim\limits_{b\to 1^-} F_2(0,b) & = & \displaystyle\int_0^{1} t\left(t^{1-k} - 1 + (1-k)(1-t) \right)^{s'-1}\,dt > 0.
	\end{array}
\]
Hence, there exist points $a^{*},b^{*}\in (0,1)$ such that $F_2\left(0,b^*\right) = F_2\left(a^{*}, a^{*}\right) = 0$. Taking into account continuity of function $F_2$ and its monotonicity in both variables we conclude that for every $a\in\left[0,a^{*}\right]$ there exists $b = \varrho(a) \in\left[a^*,b^*\right]$ such that $F_2\left(a,b\right) = 0$. Moreover, the function $\varrho$ is continuous and is decreasing on the interval $\left[0,a^*\right]$ because it has an inverse function. Finally, we observe that
\[
	\begin{array}{rcl}
		\displaystyle F_1\left(0,b^*\right) & = & \displaystyle \frac{1}{b^*}\left\{b^* \int_{0}^{1} \tau_{(s')}\left(0,b^*;t\right)\,dt\right\} \\ [10pt]
		& > & \displaystyle \frac{1}{b^*}\left\{\int_{0}^{b^*} t \cdot\tau_{(s')}\left(0,b^*;t\right)\,dt\right\} = \frac{F_2\left(0,b^*\right)}{b^*} = 0,
	\end{array}
\]
and
\[
	F_1\left(a^*,a^*\right) = \int_{a^*}^{1} \tau_{(s')} \left(a^*,a^*;t\right)\,dt< 0.
\]
Hence, there exists $a_0\in\left(0,a^*\right)$ such that $F_1\left(a_0,\varrho\left(a_0\right)\right) = 0$. The latter implies that $\left(a_0,\varrho\left(a_0\right)\right) \in M$ and satisfies system~(\ref{2nd_equation}).
\end{proof}

\begin{proof}[Proof of Corollary~\ref{Cor:Kolmogorov_inequality_R+_k_1..2}]
We set $\Omega := \frac{\omega_{k,s}}{\Gamma(2-k)}$. It is easy to check that for $f\in L^2_{\infty,s}\left(\mathbb{R}_+\right)$,
\[
  D^{k}_{-}f(0) - \int_0^{+\infty} f(t)\,d\Omega(t) = \int_0^{+\infty} \left(\mathcal{R}_{2-k}(t) - \Omega^{[1]}(t)\right) f''(t)\,dt.
\]
Moreover, $\Phi_{k,s}\in W^2_{\infty,s}\left(\mathbb{R}_+\right)$, $\Phi_{k,s}(0) = -\Phi_{k,s}\left(a_{k,s}\right) = \left\|\Phi_{k,s}\right\|_{L_\infty\left(\mathcal{R}_+\right)}$, $\Phi_{k,s}(x) = \Phi_{k,s}(0)$ for every $x\geqslant h$, $\Phi_{k,s}$ decreases on $\left[0,a_{k,s}\right]$ and increases on $\left[a_{k,s},+\infty\right)$. Hence, the functions $\Omega$ and $\Phi_{k,s}$ satisfy conditions of Corollary~\ref{Main_Corollary_Ls}.
\end{proof}

\subsection{Case $G = \mathbb{R}$, $1\leqslant s\leqslant \infty$ and $k\in(0,1)$}
\label{SubSec:Rk0..1r=2}

Let $k\in(0,1)$, $1\leqslant s\leqslant \infty$ and $s' = s/(s-1)$. For $p\in\left[0,k/(1-k)\right]$, we consider the function
\[
	\omega(p;x) = \left\{\begin{array}{ll}
		0, & x\leqslant -p, \\
		\displaystyle -(1+p)^{-1}, & x\in(-p,1], \\
		\displaystyle (1-k)x^{-k}, & x\geqslant 1.
	\end{array}\right.
\]
For $x\in\mathbb{R}$, we consider the function $\tau(p;x) := \Gamma(2-k)\mathcal{R}_{2-k}(x) - \omega^{[1]}(p;x)$ and define
\[
	\varphi_s(p;x) := \frac 12\int_{-p}^{1}t\cdot\tau_{(s')}(p;t)\,dt \displaystyle + \int_{-p}^{x}(x-t)\cdot\tau_{(s')}(p;t)\,dt,\qquad \textrm{when}\quad s>1,
\]
\[
	\varphi_1(p;x) := \left\{
		\begin{array}{ll}
		\left(1+p\right)/4, & x\leqslant -p, \\
		\left(1+p-2x\right)/4, & x\in\left(-p,1\right),\\
		-\left(1+p\right)/4, & x\geqslant 1.
		\end{array}
	\right.
\]
Below in Lemma~\ref{Lem:Existence_R_k_0..1} we shall prove that equations~(\ref{1st_equation_v1}) have a unique solution. For convenience, we denote by $p_{k,s}$, $s>1$, the solution to the first equation in~(\ref{1st_equation_v1}) and by $p_{k,1}$ the solution to the second equation in~(\ref{1st_equation_v1}). In addition, we denote the functions $\omega\left(p_{k,s};\cdot\right)$, $\tau\left(p_{k,s};\cdot\right)$ and $\varphi_s\left(p_{k,s};\cdot\right)$ by $\omega_{k,s}$, $\tau_{k,s}$ and $\varphi_{k,s}$ respectively. The graphs of functions $\omega^{[1]}(p;\cdot)$, $\tau(p;\cdot)$ and $\varphi_{k,s}$ are shown in Figure~\ref{fig:1.5}  on page~\pageref{fig:1.5}.

There holds true the following consequence of Corollaries~\ref{Main_Corollary_Ls} and~\ref{Main_Corollary_Ls_prime}.

\begin{cor}
\label{Cor:Kolmogorov_inequality_R_k_0..1}
Let $1\leqslant s\leqslant\infty$, $s' = s/(s-1)$, $k\in(0,1)$ and $h>0$. Then for every $f\in L_{\infty,s}^{2}\left(\mathbb{R}\right)$, there hold true sharp inequality
\[
	\left\|D^{k}_{-}f\right\|_{L_\infty\left(\mathbb{R}\right)} \leqslant \frac{\left\|D^{k}_{-}\Phi_{k,s}\right\|_{L_\infty\left(\mathbb{R}\right)}}{\left\|\Phi_{k,s}\right\|_{L_\infty\left(\mathbb{R}\right)}^{1-\lambda}}\, \left\|f\right\|_{L_\infty\left(\mathbb{R}\right)}^{1-\lambda} \left\|f''\right\|_{L_s\left(\mathbb{R}\right)}^{\lambda},\quad \lambda = \frac{k}{2-1/s},
\]
where $\Phi_{k,s} := \left\|\varphi''_{k,s}\right\|^{-1}_{L_s\left(\mathbb{R}\right)}\cdot \varphi_{k,s}$, $s>1$, and $\Phi_{k,1}:=\varphi_{k,1}$.
\end{cor}


We start with the proof of the following auxiliary proposition.

\begin{lem}
\label{Lem:Existence_R_k_0..1}
Let $1< s\leqslant\infty$, $s' = s/(s-1)$ and $k\in(0,1)$. Then the following equations have a unique solution on the interval $\left[0,k/(1-k)\right]$:
\begin{equation}
\label{1st_equation_v1}
	Z_s(p):= \int_{-p}^{1} \tau_{(s')}(p;t)\,dt = 0 \;\;\textrm{and}\;\; Z_1(p):=k^{k}(1-k)^{1-k}(1+p) - (2k)^{1-k} = 0.
\end{equation}
\end{lem}

We remark that for particular values of $s$ we can find $p_{k,s}$ explicitly, e.g. $p_{k,\infty} = 1-2^{-k/(1-k)}$ and $p_{k,2} = k/(2-k)$.

\begin{proof} The fact that equation $Z_1(p) = 0$ has a unique solution on the interval $\left[0,k/(1-k)\right]$ is trivial. To prove that equation $Z_s(p) = 0$ also has a unique solution on the same interval we observe that $Z_s$ is continuous and strictly decreases on the interval $\left[0,k/(k-1)\right]$, and attains values of opposite signs at points $0$ and $k/(k-1)$. Thus equation $Z_s(p) = 0$ has a unique solution on $\left[0,k/(1-k)\right]$.
\end{proof}

\begin{proof}[Proof of Corollary~\ref{Cor:Kolmogorov_inequality_R_k_0..1}]
We set $\Omega := \frac{\omega_{k,s}}{\Gamma(2-k)}$.
Similarly to Subsection~\ref{SubSec:R+k0..1r=2} we can check that for every $f\in L_{\infty,s}^2\left(\mathbb{R}\right)$, equality~(\ref{relation}) holds true. Moreover, if $s>1$, we see that $\Phi_{k,s}$ decreases on $\mathbb{R}$, $\Phi_{k,s}\in W^2_{\infty,s}\left(\mathbb{R}\right)$, $\Phi_{k,s}\left(-p_{k,s}\right) = - \Phi_{k,s}(1) = \left\|\Phi_{k,s}\right\|_{L_{\infty}\left(\mathbb{R}\right)}$. Hence, the functions $\Omega$ and $\Phi_{k,s}$ satisfy conditions of Corollary~\ref{Main_Corollary_Ls}. In turn, for $s=1$ we can check that the functions $\Omega$ and $\Phi_{k,1}$ satisfy conditions of Corollary~\ref{Main_Corollary_Ls_prime}. \end{proof}

\subsection{Case $G=\mathbb{R}$, $1<s\leqslant \infty$ and $k\in(1,2-1/s)$}
\label{SubSec:Rk1..2r=2}

Let $1<s\leqslant\infty$, $s'=s/(s-1)$ and $k\in\left(1,2-1/s\right)$. Consider the set $S := M\times \left[0, +\infty\right)$, where the set $M$ was defined in Subsection~\ref{SubSec:R+k1..2r=2}, and for every $(a,b,p)\in S$, we define
\[
	\omega(a,b,p;x) = \left\{\begin{array}{ll}
		0, & x\leqslant -p,\\
		\displaystyle \frac{1-b+\left(1-b^{1-k}\right)(a-1)}{(1-b)(a+p)}, & x\in[-p,a], \\
		\displaystyle (1-b)^{-1}\cdot\left(b^{1-k}-1\right), & x\in[a,1], \\
		(1-k)x^{-k}, & x\geqslant 1.
	\end{array}\right.
\]
For $x\in\mathbb{R}$, we consider functions $\tau(a,b,p;x) :=\Gamma(2-k)\cdot\mathcal{R}_{2-k}(x) - \omega^{[1]}(a,b,p;x)$ and
\[
	\displaystyle \varphi(a,b,p;x) := \displaystyle \frac 12\int_{-p}^{a}t\cdot \tau_{(s')}(a,b,p;t)\,dt + \int_{-p}^{x} (x-t)\cdot\tau_{(s')}(a,b,p;t)\,dt.
\]
Below in Lemma~\ref{Lem:Existence_R_k_1..2} we shall prove that the system of equations~(\ref{3rd_equation}) has at least one solution on $S$. Let $\left(a_{k,s},b_{k,s},p_{k,s}\right)$ be one of such solutions  and, for simplicity of notations, we denote the functions $\omega\left(a_{k,s},b_{k,s},p_{k,s};\cdot\right)$, $\tau\left(a_{k,s},b_{k,s},p_{k,s};\cdot\right)$ and $\varphi\left(a_{k,s},b_{k,s},p_{k,s};\cdot\right)$ by $\omega_{k,s}$, $\tau_{k,s}$ and $\varphi_{k,s}$ respectively. The graphs of functions $\omega^{[1]}(a,b,p;\cdot)$, $\tau(a,b,p;\cdot)$ and $\varphi_{k,s}$ are shown in Figure~\ref{fig:1.7} on page~\pageref{fig:1.7}.

Next, we set $\Phi_{k,s} := \left\|\varphi''_{k,s}\right\|_{L_s\left(\mathbb{R}\right)}^{-1}\cdot \varphi_{k,s}$. There holds true the following consequence of Corollary~\ref{Main_Corollary_Ls}.

\begin{cor}
\label{Cor:Kolmogorov_inequality_R_k_1..2}
Let $1<s\leqslant\infty$, $s' = s/(s-1)$, $k\in\left(1,2-1/s\right)$, and $h>0$. Then for every function $f\in L_{\infty,s}^2(\mathbb{R})$, there hold true sharp inequalities
\[
  \left\|D^{k}_{-}f\right\|_{L_\infty\left(\mathbb{R}\right)} \leqslant \frac{\left\|D^{k}_{-}\Phi_{k,s}\right\|_{L_\infty\left(\mathbb{R}\right)}}{\left\|\Phi_{k,s}\right\|_{L_\infty\left(\mathbb{R}\right)}^{1-\lambda}}\,\left\|f\right\|^{1-\lambda}_{L_\infty\left(\mathbb{R}\right)} \left\|f''\right\|^{\lambda}_{L_s\left(\mathbb{R}\right)},\qquad \lambda = \frac{k}{2-1/s}.
\]
\end{cor}


We start with the proof of the following auxiliary

\begin{lem}
\label{Lem:Existence_R_k_1..2}
The following system of equations has a solution on $S$:
\begin{equation}
\label{3rd_equation}
	\left\{ \begin{array}{l}
		\displaystyle Z_1(a,b,p) := \int_{-p}^{a}\tau_{(s')}(a,b,p;t)\,dt = 0, \\[10pt]
		\displaystyle Z_2(a,b,p) := \int_{a}^{1} \tau_{(s')}(a,b,p;t)\,dt = 0, \\[10pt]
		\displaystyle Z_3(a,b,p) := \int_{-p}^{1} t\cdot \tau_{(s')}(a,b,p;t)\,dt = 0,
	\end{array}\right.
\end{equation}
\end{lem}

For particular values of $s$ the triple  $\left(a_{k,s},b_{k,s},p_{k,s}\right)$ can be found explicitly, e.g. $ p_{k,\infty} = a_{k,\infty} = 1/3$ and $b_{k,\infty} = 2/3$.

\begin{proof} First, we observe that function $Z_2$ is continuous on $S$, strictly decreases in variable $a$, strictly increases in variable $b$ and is constant in variable $p$. In addition, for every $a\in(0,1)$, $Z_2\left(a,a,0\right) <0$ and $\lim\limits_{b\to 1-0} Z_2\left(a,b,0\right) > 0$. The latter and monotonicity of $Z_2$ in variable $b$ imply that there exists a strictly increasing function $\gamma : (0,1)\to\mathbb{R}$ such that $Z_2\left(a,\gamma(a),0\right) = 0$. Moreover, continuity of function $\gamma$ follows from its monotonicity and continuity of function $Z_2$.


Next, we consider the function $Z_1$. Evidently, $Z_1$ is continuous on $S$, strictly increases in variables $a$ and $b$, and strictly decreases in variable $p$. Since for every $a\in(0,1)$, $Z_1(a,\gamma(a),0) > 0$ and $\lim\limits_{p\to -\infty} Z_1\left(a,\gamma(a),p\right) = -\infty$, we conclude that there exists a function $\delta:(0,1)\to\mathbb{R}$ such that for every $a\in(0,1)$, $Z_1\left(a,\gamma(a),\delta(a)\right) = 0$. Since $Z_1$ is continuous on $S$ and $\delta$ is monotone, we conclude that $\delta$ is also continuous on $(0,1)$. Therefore, for every $a\in(0,1)$, we have $Z_1\left(a,\gamma(a),\delta(a)\right) = Z_2\left(a,\gamma(a),\delta(a)\right) = 0$. Now, we set $b^* = \lim\limits_{a\to +0} \gamma(a)$. Since,
\[
	\begin{array}{rcl}
		\displaystyle\lim\limits_{a\to 0^+} Z_2\left(a,\frac 12,0\right) & = & \displaystyle \int_0^{1/2} \left(t^{1-k} - 1 + 2\left(1-2^{k-1}\right)(1-t)\right)^{s'-1}\,dt \\[10pt]
		& & \qquad \displaystyle - \int_{1/2}^{1} \left(t^{1-k} - 1 + 2\left(1-2^{k-1}\right)(1-t)\right)^{s'-1}\,dt >0,
	\end{array}
\]
and
\[
	\lim\limits_{a\to 0^+} Z_2\left(a,\frac 12,0\right) = \int_0^{1/2} \left(t^{1-k} - t\right)^{s'-1}\,dt - \int_{1/2}^{1} \left(t^{1-k} - t\right)^{s'-1}\,dt >0,
\]
we conclude that $b^* \in(0,1)$.

Finally, note that $Z_3$ is also continuous on $S$ function, and
\[
	\begin{array}{rcl}
		\displaystyle\lim\limits_{a\to 0^+} Z_3\left(a,\gamma(a),\delta(a)\right) & = & \displaystyle Z_3\left(+0, b^*, 0\right) = \int_0^{b^*} t\cdot \tau^{s'-1}_{k}\left(+0,b^*,0\right)\,dt \\ [10pt]
		& & \displaystyle\qquad - \int_{b^*}^1 t\cdot \tau^{s'-1}_{k}\left(+0,b^*,0\right)\,du \\[10pt]
		& < & \displaystyle b^*\left[\int_0^{b^*}\tau_{k}^{s'-1}(a,b,p;t)\,dt - \int_{b^*}^1 \tau^{s'-1}_{k}(a,b,p;t)\,dt \right]\\[10pt]
		& = & \displaystyle b^* Z_2\left(+0,b^*,0\right) = 0,
	\end{array}
\]
\[
	\begin{array}{rcl}
		\displaystyle \lim\limits_{a\to 1^-} Z_3\left(a,\gamma(a),\delta(a)\right) & > & \displaystyle \int_0^1 t \left(t^{1-k} - 1\right)^{s'-1}\,dt > 0.
	\end{array}
\]
Hence, there exists $a\in(0,1)$ such that $Z_3\left(a,\gamma(a),\delta(a)\right) = 0$.
\end{proof}

\begin{proof}[Proof of Corollary~\ref{Cor:Kolmogorov_inequality_R_k_1..2}]
We set $\Omega := \frac{\omega_{k,s}}{\Gamma(2-k)}$. It is easy to check that relation~(\ref{relation}) holds true for every $f\in L^2_{\infty,s}\left(\mathbb{R}_+\right)$.
Moreover, $\Phi_{k,s}\in W^2_{\infty,s}\left(\mathbb{R}\right)$, $\Phi_{k,s}\left(-p_{k,s}\right) = -\Phi_{k,s}\left(a_{k,s}\right) = \left\|\Phi_{k,s}\right\|_{L_\infty\left(\mathbb{R}\right)}$,  $\Phi_{k,s}(x) = \left\|\Phi_{k,s}\right\|_{L_\infty\left(\mathbb{R}\right)}$ for every $x\geqslant 1$, $\Phi_{k,s}$ is non-increasing on $\left(-\infty,a_{k,s}\right]$ and non-decreasing on $\left[a_{k,s},+\infty\right)$. Hence, the functions $\Omega$ and $\Phi_{k,s}$ satisfy conditions of Corollary~\ref{Main_Corollary_Ls} which finishes the proof.
\end{proof}

\section{Applications}
\label{Sec:Applications}

In this section we consider applications of results of this paper. We devote Subsection~\ref{SubSec:Kolm} to consequences related to the Kolmogorov problem for three numbers and in Subsection~\ref{SubSec:Hadamard} we obtain sharp Kolmogorov type inequalities for the weighted norms of fractional powers of operator $\mathscr{D} = x\frac{d}{dx}$ delivered by the Hadamard fractional derivative.

\subsection{The Kolmogorov problem for three numbers}
\label{SubSec:Kolm}

Let $G = \mathbb{R}$ or $G = \mathbb{R}_+$, $1\leqslant p,q,s\leqslant\infty$, $r\in\mathbb{N}$ and $k\in\left(0,r\right)\setminus \mathbb{N}$. The Kolmogorov problem for three numbers (see~\cite{Kolm_85}) consists in finding necessary and sufficient conditions on three positive numbers $M_{0}$, $M_{k}$, $M_{r}$, that guarantee existence of a function $f\in L_{p,s}^{r}(G)$ satisfying equalities
\begin{equation}
\label{Kolmogorov_condition}
	\left\|f\right\|_{L_{p}(G)} = M_{0},\qquad \left\|D^{k}_{-}f\right\|_{L_{q}(G)} = M_{k},\qquad \left\|f^{(r)}\right\|_{L_{s}(G)} = M_{r}.
\end{equation}
For the overview of known results in this direction, we refer the reader to~\cite{Kolm_85,BKKP-kn} and references therein. Using similar arguments and combining them with results of the previous section we deduce the following consequences.

\begin{thm}
\label{Thm:Kolmogorov_problem_s>1}
Let $G=\mathbb{R}$ or $G = \mathbb{R}_+$, $1< s\leqslant\infty$, $p=q=\infty$, $k\in\left(0,2-1/s\right)$, $r=2$, and $M_0$, $M_k$, $M_2$ be positive numbers. Assume that $K$ is a sharp constant in inequality~(\ref{multiplicative_inequality_fractional}) and there exists a non-negative function turning~(\ref{multiplicative_inequality_fractional}) into equality. Then there exists a function $f\in L^{2}_{\infty,s}(G)$ satisfying equalities~(\ref{Kolmogorov_condition}) if and only if $M_k \leqslant K \cdot M_0^{1-\lambda} M_2^{\lambda}$ where $\lambda = \frac{k}{2-1/s}$.
\end{thm}

\begin{thm}
\label{Thm:Kolmogorov_problem_s=1}
Let $G=\mathbb{R}$ or $G = \mathbb{R}_+$, $s=1$, $p=q=\infty$, $k\in\left(0,2-1/s\right)$, $r=2$, and $M_0$, $M_k$, $M_2$ be positive numbers. Assume that $K$ is a sharp constant in inequality~(\ref{multiplicative_inequality_fractional}). Then there exists a function $f\in L^{r}_{\infty,s}(G)$ satisfying equalities~(\ref{Kolmogorov_condition}) if and only if $M_k < K \cdot M_0^{1-\lambda} M_2^{\lambda}$ where $\lambda = \frac{k}{2-1/s}$.
\end{thm}

\subsection{Sharp Kolmogorov type inequalities for the Hadamard fractional derivatives}
\label{SubSec:Hadamard}

Let $\mathscr{D}$ be the operator mapping every differentiable function $f:\mathbb{R}_+\to\mathbb{R}$ into the function $\mathscr{D}f(x) = xf'(x)$, $x\in\mathbb{R}_+$, i.e. $\mathscr{D} = x\frac{d}{dx}$. The fractional power, $k\in\mathbb{R}_+\setminus\mathbb{N}$, of operator $\mathscr{D}$ is delivered by the Hadamard fractional differentiation operator $\mathscr{D}_{\pm}^{k}$ (see~\cite[\S 18]{Sam}) which is defined as follows: for $f:\mathbb{R}_+\to\mathbb{R}$ and $x\in\mathbb{R}_+$,
\[
	\begin{array}{rcl}
		\displaystyle\mathscr{D}^{k}_{+}f(x) &=& \displaystyle \frac{1}{\varkappa(k,r)}\int_0^1 \sum\limits_{m=0}^{r}(-1)^m \left(r \atop m\right) f\left(u^m x\right)\,\frac{du}{u|\ln{u}|^{1+k}},\\
		\displaystyle\mathscr{D}^{k}_{-}f(x) &=& \displaystyle \frac{1}{\varkappa(k,r)}\int_1^{+\infty} \sum\limits_{m=0}^{r}(-1)^m \left(r \atop m\right) f\left(u^m x\right)\,\frac{du}{u|\ln{u}|^{1+k}},
	\end{array}
\]
where $r\in\mathbb{N}$, $r>k$, and $\varkappa(k,r)$ was defined in~(\ref{difference}). Some Kolmogorov type inequalities for the Hadamard fractional derivatives were considered in paper~\cite{Bab_Parf_09}

For an arbitrary function $f:\mathbb{R}_+\to\mathbb{R}$, let us define the function $g:\mathbb{R}\to\mathbb{R}$ as follows: $g(t) = f\left(e^{t}\right)$, $t\in\mathbb{R}$. Then for every $x\in\mathbb{R}_+$, we have $\mathscr{D}^{k}_{\pm}f(x) = D^{k}_{\pm}g(\ln{x})$. As a result, $\left\|\mathscr{D}^{k}_{\pm}f\right\|_{L_{\infty}\left(\mathbb{R}_+\right)} = \left\|D^{k}_{\pm}g\right\|_{L_{\infty}(\mathbb{R})}$ and for every $1\leqslant s < \infty$,
\[
	\int_0^{+\infty} \left|\mathscr{D}^{k}_{\pm}f(x)\right|^{s}\,\frac{dx}{x} = \left\|D^{k}_{\pm}g\right\|_{L_{s}(\mathbb{R})}^s.
\]
The latter formula allows deducing sharp Kolmogorov type inequalities for the weighted $L_s$-norms of the Hadamard fractional derivatives from sharp Kolmogorov type inequalities for $L_s$-norms of the Marchaud fractional derivatives. Let us present rigorous statements. For $1\leqslant s\leqslant \infty$, by $\mathscr{L}_s$ we denote the space of functions $f:\mathbb{R}_+\to\mathbb{R}$ endowed with the norm
\[
	\|f\|_{\mathscr{L}_s} = \left\{\begin{array}{ll}
		\displaystyle \left(\int_0^{+\infty} |f(x)|^{s}\,\frac{dx}{x}\right)^{1/s}, & 1\leqslant s<\infty, \\ [8pt]
		\displaystyle \|f\|_{L_{\infty}\left(\mathbb{R}_+\right)}, & s=\infty.
	\end{array}\right.
\]
For $r\in\mathbb{N}$, let $\mathscr{L}_{\infty,s}^r$ be the space of functions $f\in L_{\infty}\left(\mathbb{R}_+\right)$ such that $f^{(r-1)}$ is locally absolutely continuous on $\mathbb{R}_+$, and $f^{(r)}\in\mathscr{L}_{s}$.

From above arguments we conclude that the following proposition holds true.

\begin{thm}
\label{Thm:connection_between_Hadamard_and_Marchaud}
Let the Kolmogorov type inequality~(\ref{multiplicative_inequality}) with sharp constant $K$ hold true for some collection of parameters $1\leqslant p,q,s\leqslant \infty$, $r\in\mathbb{N}$, $k\in(0,r)$ and $\mu,\lambda\in\mathbb{R}$, $0\leqslant \mu = 1-\lambda \leqslant 1$. Then for the same collection of parameters and for every $f\in \mathscr{L}_{p,s}^{r}$, there holds true sharp inequality
\[
  \left\|\mathscr{D}_{\pm}^{k}f\right\|_{\mathscr{L}_q} \leqslant K\, \left\|f\right\|_{\mathscr{L}_p}^{\mu} \left\|\mathscr{D}^rf\right\|_{\mathscr{L}_s}^{\lambda}.
\]
\end{thm}

Combining Theorem~\ref{Thm:connection_between_Hadamard_and_Marchaud} with Corollaries~\ref{Cor:Kolmogorov_inequality_R_k_0..1} and~\ref{Cor:Kolmogorov_inequality_R_k_1..2} we obtain

\begin{cor}
\label{Cor:Multiplicative_Hadamard_k_0..1}
Let $1\leqslant s\leqslant\infty$, $s' = s/(s-1)$, $k\in\left(0,2-1/s\right)$ and $\lambda = k/(2-1/s)$. Then for every function $f\in \mathscr{L}_{\infty,s}^{2}$, there holds true sharp inequality
\[
  \left\|\mathscr{D}^{k}_{\pm}f\right\|_{\mathscr{L}_{\infty}} \leqslant \frac{\left\|D^{k}_{-}\Phi_{k,s}\right\|_{L_{\infty}(\mathbb{R})}}{\left\|\Phi_{k,s}\right\|^{1-\lambda}_{L_{\infty}(\mathbb{R})}}\, \left\|f\right\|_{\mathscr{L}_{\infty}}^{1-\lambda} \left\|\mathscr{D}^{2}f\right\|_{\mathscr{L}_{s}}^{\lambda},
\]
where the function $\Phi_{k,s}$ was defined in Subsections~\ref{SubSec:R+k0..1r=2}--\ref{SubSec:Rk1..2r=2}.
\end{cor}

\newpage

\section*{Figures}

\begin{figure}[h!]
	\includegraphics[width=1.5in, height=1.1in]{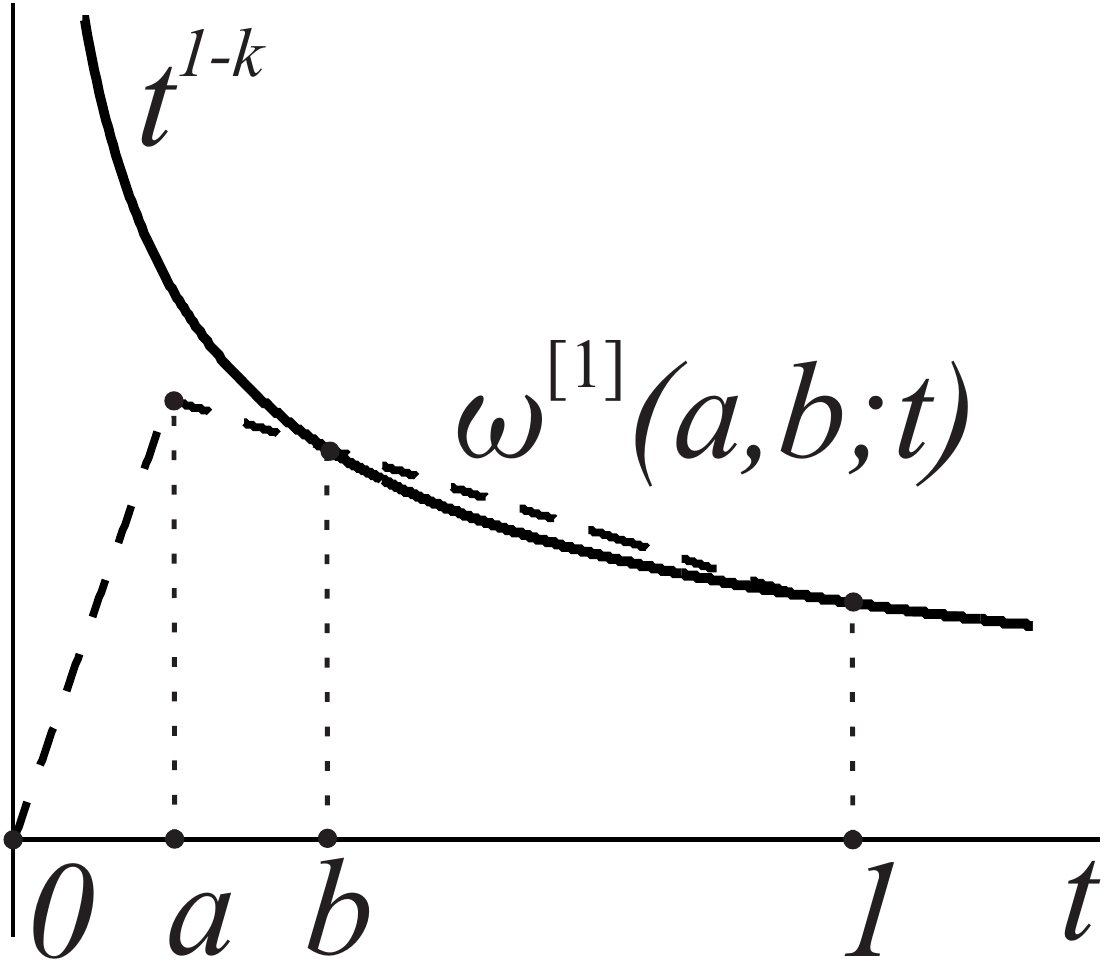}
	\includegraphics[width=1.5in, height=1.1in]{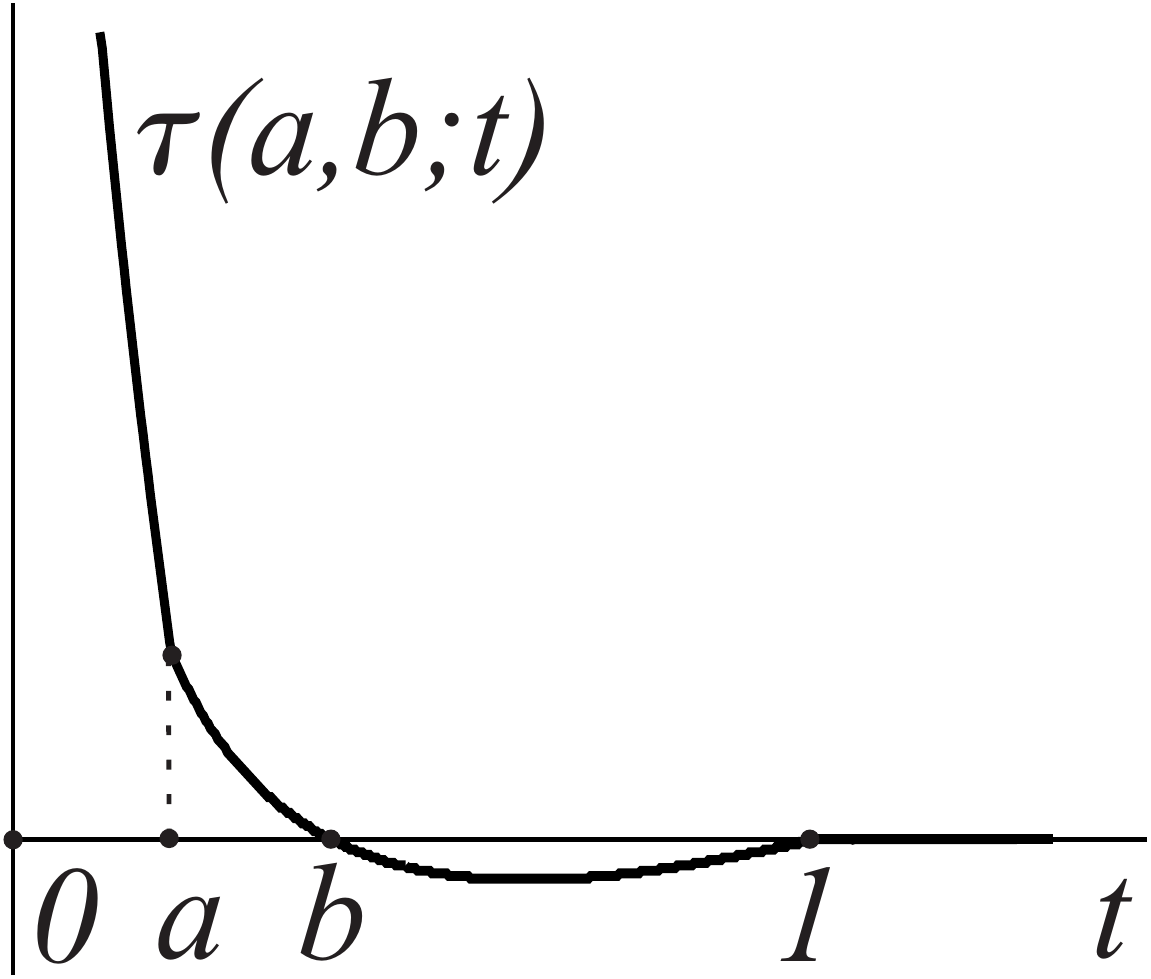}
	\includegraphics[width=1.5in, height=1.1in]{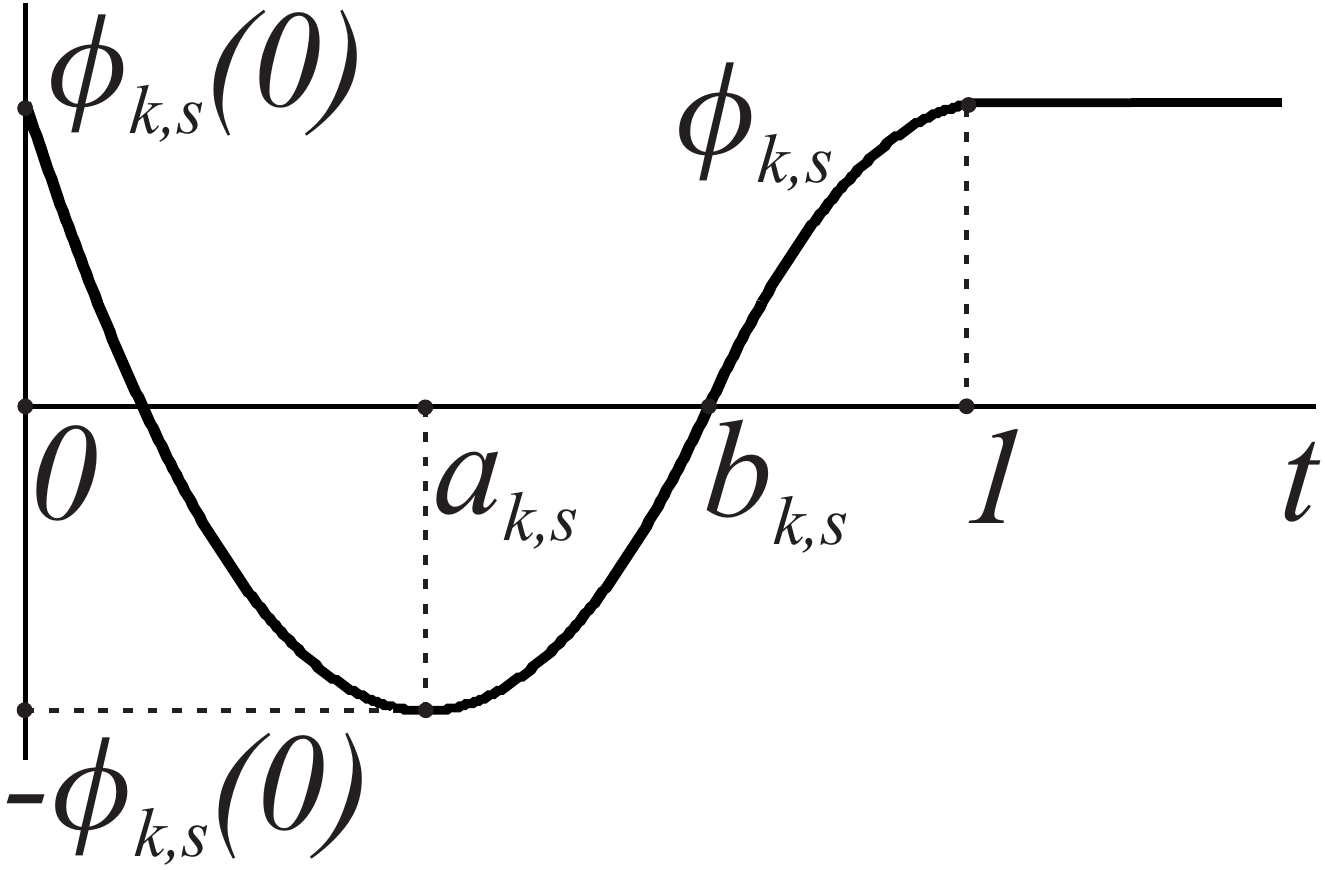}
	\caption{Functions from left to right: $\Gamma(2-k)\cdot\mathcal{R}_{2-k}$ and $\omega^{[1]}(a,b;\cdot)$; $\tau(a,b;\cdot)$; $\varphi_{k,s}$}
	\label{fig:1.3}
\end{figure}

\begin{figure}[h!]
	\includegraphics[width=1.5in, height=1.1in]{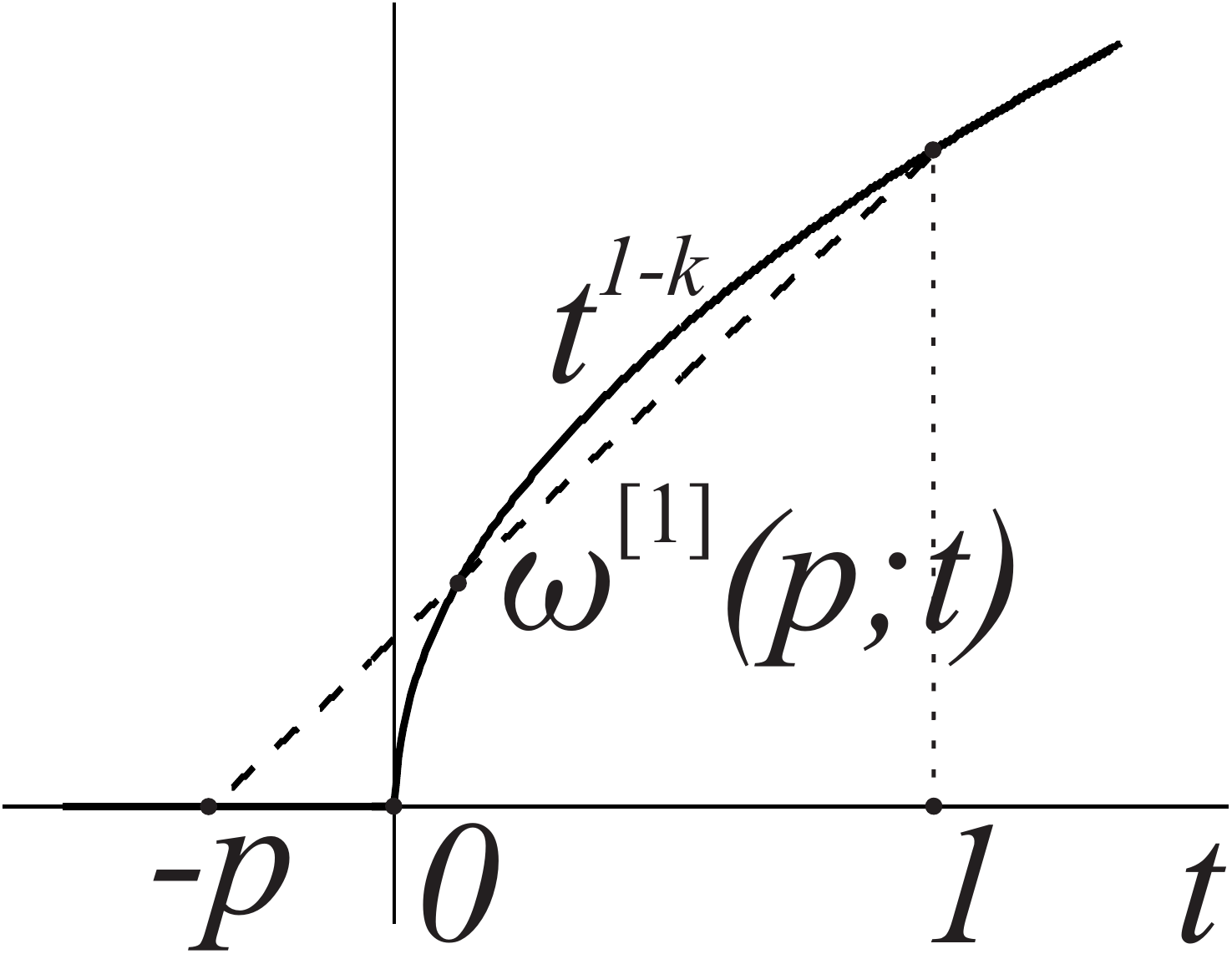}
	\includegraphics[width=1.5in, height=1.1in]{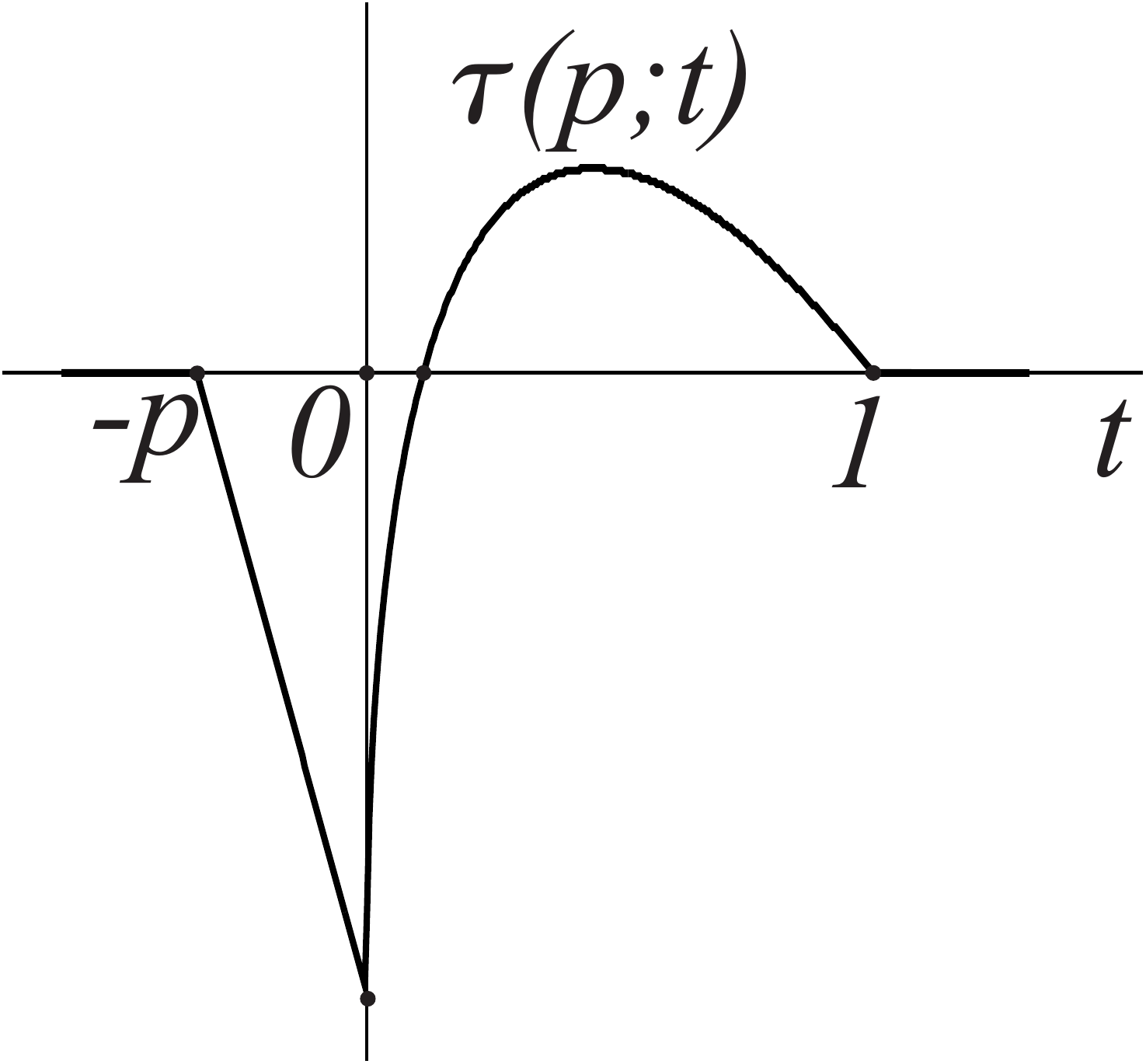}
	\includegraphics[width=1.5in, height=1.1in]{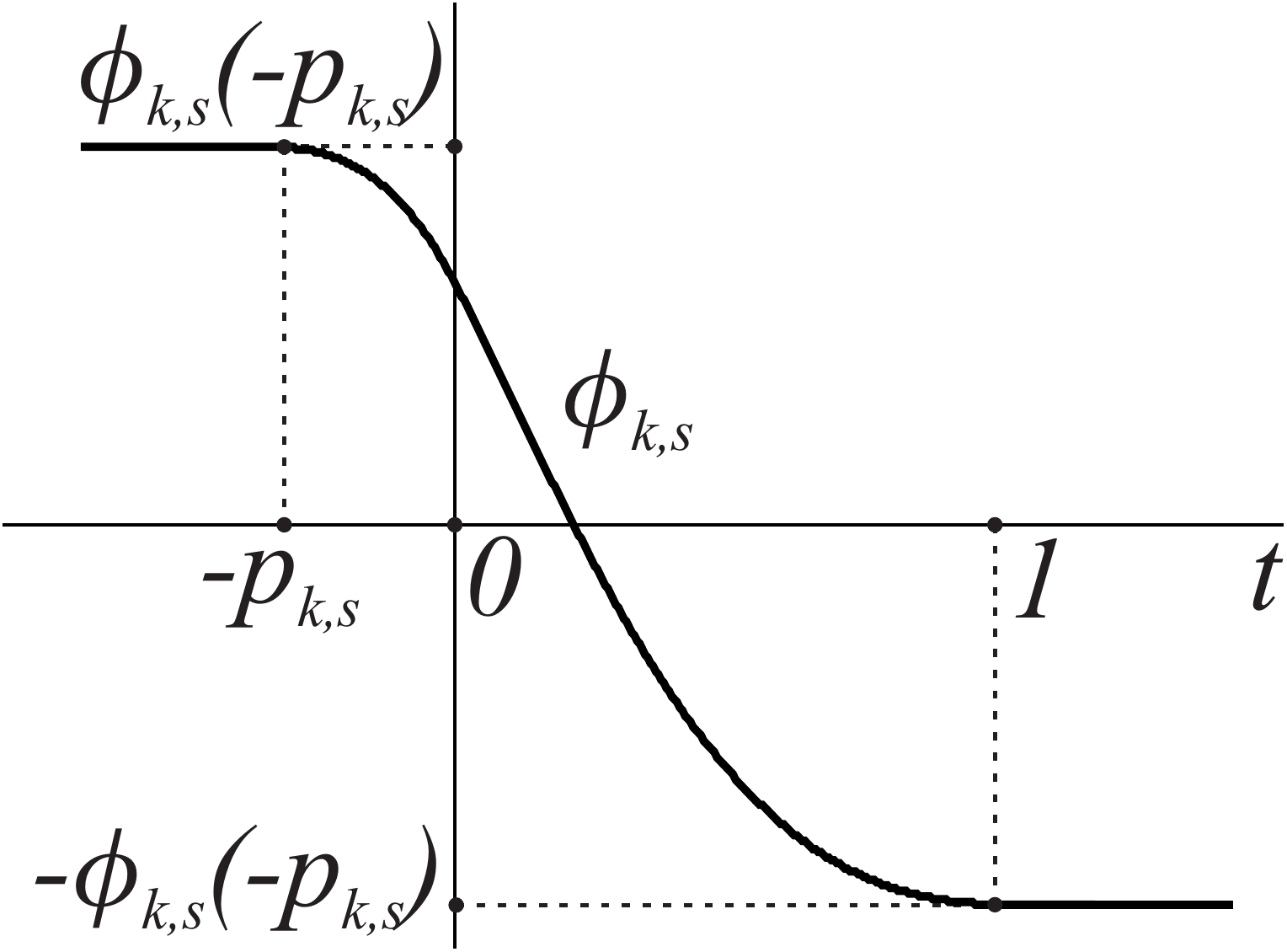}
	\caption{Functions from left to right: $\Gamma(2-k)\cdot\mathcal{R}_{2-k}$ and $\omega^{[1]}(p;\cdot)$; $\tau(p;\cdot)$; $\varphi_{k,s}$}
	\label{fig:1.5}
\end{figure}
\begin{figure}[h!]
	\includegraphics[width=1.5in, height=1.1in]{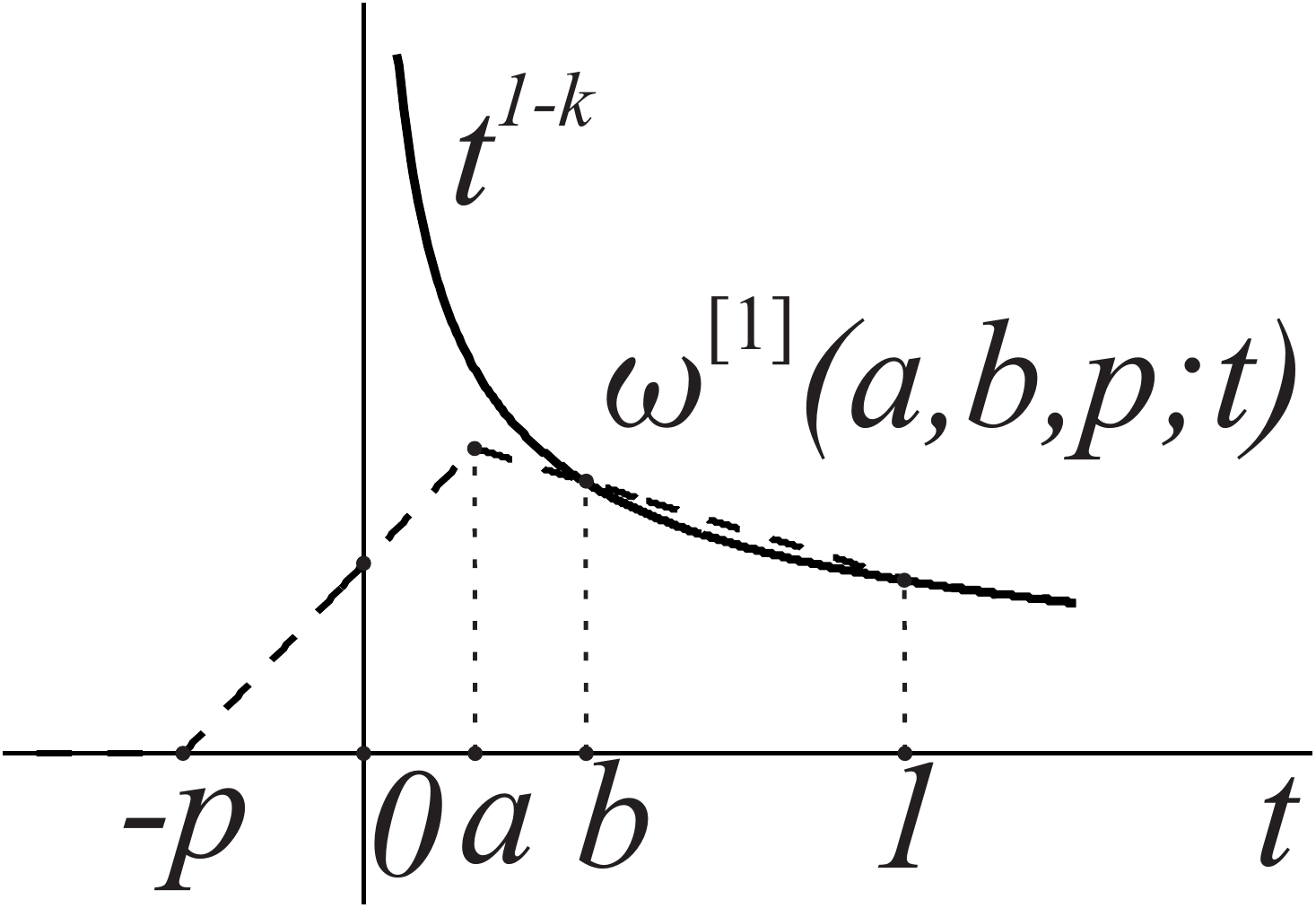}
	\includegraphics[width=1.5in, height=1.1in]{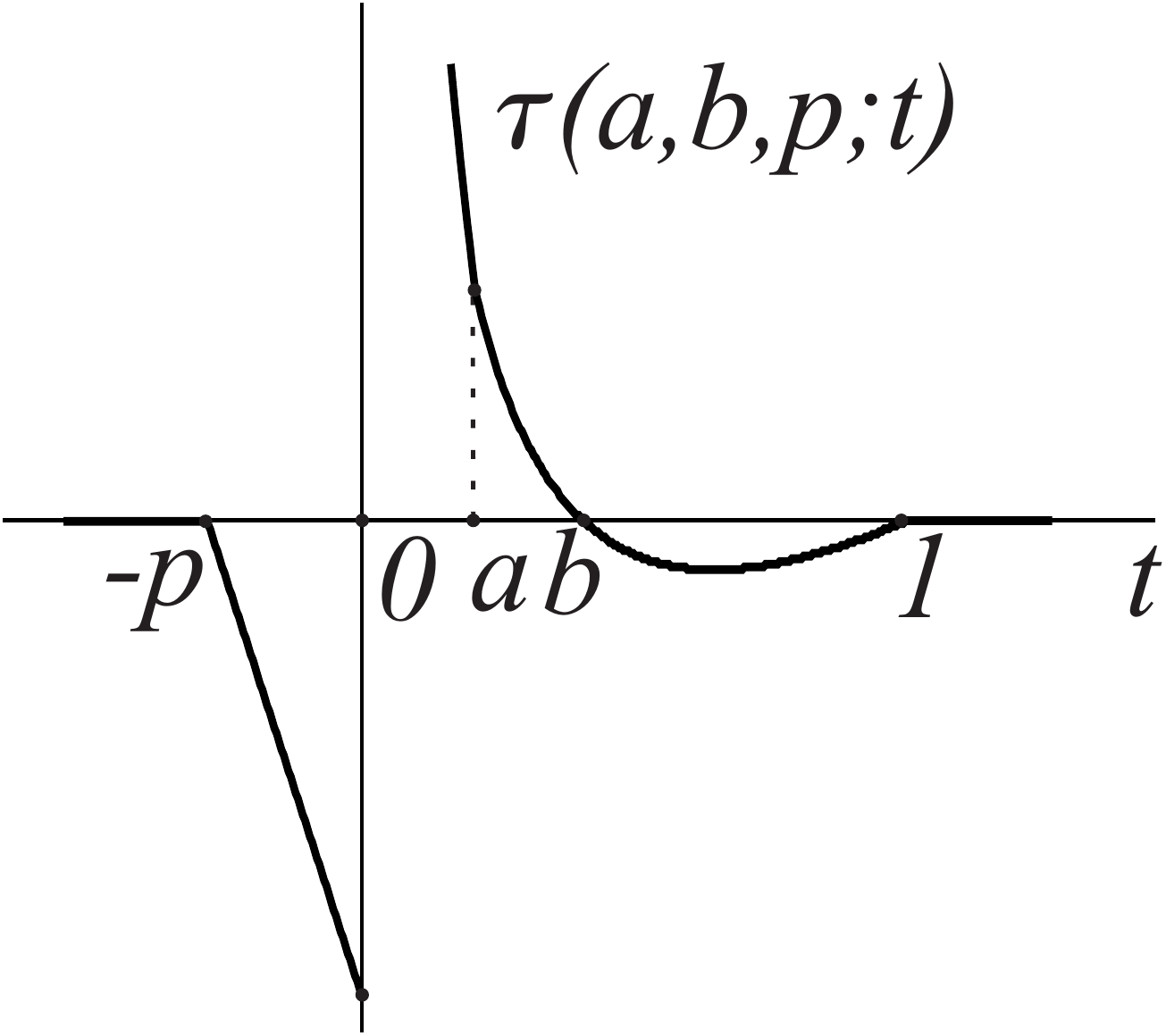}
	\includegraphics[width=1.5in, height=1.1in]{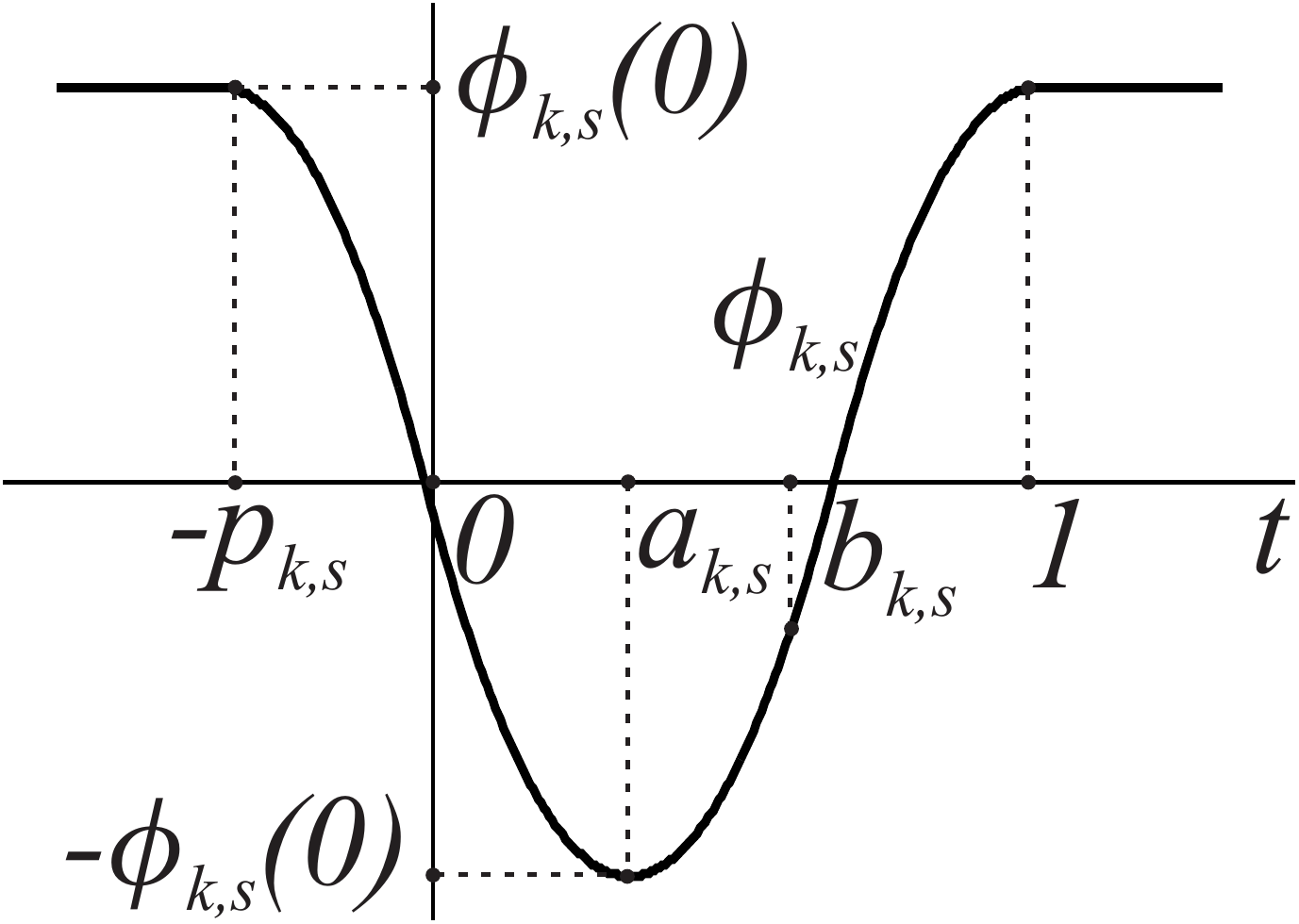}
	\caption{Functions from left to right: $\Gamma(2-k)\cdot\mathcal{R}_{2-k}$ and $\omega^{[1]}(a,b,p;\cdot)$; $\tau(a,b,p;\cdot)$; $\varphi_{k,s}$}
	\label{fig:1.7}
\end{figure}

\bibliographystyle{bmc-mathphys} 
\bibliography{BCPS_to_arxiv}      

\end{document}